%% file: preprint.tex
\documentclass[12pt,a4paper]{article}

\usepackage{hyperref}
\usepackage{csquotes}
\usepackage{lipsum}
\input{preamble.tex}

\newcommand{\N}{\mathbb{N}}
\newcommand{\R}{\mathbb{R}}
\newcommand{\C}{\mathbb{C}}

\begin{document}
\selectlanguage{english}

\title{On 3d dipolar Bose-Einstein condensates involving quantum fluctuations and three-body interactions}
\author{Yongming Luo\thanks{Institut f\"{u}r Mathematik, Universit\"{a}t Kassel, 34132 Kassel, Germany}
\ and \
Athanasios Stylianou$^*$
}
\maketitle

\begin{abstract}
We study the following nonlocal mixed order Gross-Pitaevskii equation
$$i\,\partial_t \psi=-\frac{1}{2}\,\Delta \psi+V_{ext}\,\psi+\lambda_1\,|\psi|^2\,\psi+\lambda_2\,(K*|\psi|^2)\,\psi+\lambda_3\,|\psi|^{p-2}\,\psi,$$
where $K$ is the classical dipole-dipole interaction kernel, $\lambda_3>0$ and $p\in(4,6]$; the case $p=6$ being energy critical. For $p=5$ the equation is considered  currently as the state-of-the-art model for describing the dynamics of dipolar Bose-Einstein condensates (Lee-Huang-Yang corrected dipolar GPE). We prove existence and nonexistence of standing waves in different parameter regimes; for $p\neq 6$ we prove global well-posedness and small data scattering.
\end{abstract}

\footnotetext[1]{\textbf{Keywords:} dipolar Bose-Einstein condensates, nonlocal Gross-Pitaevskii equation, concentration-compactness}
\footnotetext[2]{\textbf{2010 AMS Subject Classification:} 35Q55, 49J35, 35B09}

\section{Introduction}
The static and dynamic properties of a Bose-Einstein condensate (BEC) can be studied through an effective mean field equation known as the \textit{Gross-Pitaevskii equation} (GPE)
\begin{equation}
 \label{GPE} i\hbar\,\partial_t \Psi=-\frac{\hbar^2}{2m}\,\Delta \Psi+g\,|\Psi|^2\,\Psi+V_{ext}\Psi,
\end{equation}
a variant of the famous nonlinear Schr\"odinger equation. Here $\Psi$ is the BEC wavefunction, $V_{ext}$ is an external potential needed to keep the BEC in place (the trapping potential), $|g| = 4\pi\hbar^2 N |a|/m$, $N$ is the total number of particles in the condensate, $m > 0$ denotes the mass of a particle and $a\in\mathbb R$ its corresponding scattering length. The latter can be tuned to be either positive or negative, corresponding to an repulsive (defocusing) or attractive (focusing) quantum pressure. Moreover, the wavefunction $\Psi$ is normalized so that $\lnorm \Psi(t)\rnorm_2=1$ for all $t$. This is the classical model for BECs; for more details on mean field theory see for example \cite{KevrekidisEtAl2015} and references therein.

BECs made of dipolar (i.e. highly magnetic) atoms (e.g. chromium, dysprosium, erbium etc) were first created in the mid 2000's by the group of T. Pfau in Stuttgart. For such gases, a dipole-dipole interaction between the atoms becomes important. This action is long ranged and anisotropic and gives rise to a rich array of new phenomena (\cite{LahayeEtAl2009}). However, recent observations have been made (\cite{KadauEtAl2016,SchmittEtAl2016}) during experiments with dysprosium, not accounted for by the standard mean field theory corresponding to \eqref{GPE}. These experiments produced a stable droplet crystal, similar to ones observed in classical ferrofluids. In contrast to the observation, mean field theory predicted the collapse of these droplets to extremely high densities. It was then first suggested that the repulsive dipolar interaction is responsible for the stabilization of the condensate. However, mean field calculations have shown that this is not the case, i.e. adding a dipolar interaction term does not necessarily stabilize the condensate (for more details see \cite{KadauEtAl2016} and references within).

Thus it was suggested to modify \eqref{GPE} by adding a nonlocal (convolution integral) term and a nonlinear higher order term, respectively modelling the long range dipole-dipole interactions and the beyond mean field quantum fluctuations (the so-called Lee-Huang-Yang correction); see also \cite{BissetEtAl2016,Malomed2018} and references therein.

The \textit{extended dipolar Gross-Pitaevskii equation} (edGPE) reads:
\begin{equation}
 \label{Dimensional_edGPE} i\hbar\,\partial_t \Psi=-\frac{\hbar^2}{2m}\,\Delta \Psi+g\,|\Psi|^2\,\Psi+V_{ext}\Psi+(V_{dip}*|\Psi|^2)\,\Psi+g_{p}\,|\Psi|^{p-2}\,\Psi,
\end{equation}
where the potential $V_{dip}:\mathbb R^3\strongly \mathbb R$ describes the dipole-dipole interaction. We consider only the case $g_p>0$. For $p=5$ we obtain the Lee-Huang-Yang correction; the Lee-Huang-Yang coefficient $g_{5}$ is then always positive (see \cite{Malomed2018}). The case $p=6$ corresponds to the energy critical case and, with a positive coefficient $g_6$, may describe short-range conservative three-body interactions (see \cite{Blakie2016}); the case $g_6<0$ models three-body losses (\cite{MetzEtAl2009}) and due to its high complexity lies out of the scope of this paper.

This type of pattern formation is a very interesting phenomenon: similar to the so-called Rosensweig instability of ferrofluids (see  e.g. \cite{GrovesLloydEtAl2017,PariniStylianou2018} and references therein), it appears in a system as a stable state. On the other hand, it has been mostly pattern formation at systems driven far from equilibrium (e.g. Rayleigh-B\'enard convection, Taylor-Couette flow or current instabilities) that has been the usual case of study (\cite{RichterLange2009}).

Moreover, after experimental observations (\cite{SchmittEtAl2016}) there has been numerical evidence (\cite{BaillieEtAl2016}; concerning edGPE theory described above) that the aforementioned patterns remain stable even after the trapping potential is turned off. This is another surprising feature that is not present in the classical GPE theory and motivates the setting in this paper.

Here is a summary of our results: for $p\neq6$ we prove that the stabilizing effect of the highest order term (the Lee-Huang-Yang correction for $p=5$) is indeed very strong, so that \eqref{Dimensional_edGPE} is well-posed in $H^1(\mathbb R^3;\C)$, i.e. it possesses a unique (up to invariances) global in time solution, that scatters for small initial data. Moreover, for $p\in(4,6]$, we prove the existence of a parameter regime where the solutions are standing waves; we prove that standing waves do not exist outside this regime. It is quite cumbersome to give an explicit definition of this regime, however, we are able to give some estimates.

Local well-posedness for the time-dependent problem is proven by a standard fixed-point argument using Strichartz estimates (Kato's method). We then give uniform in time bounds for the local solution to extend it to the whole real line. For the energy critical case ($p=6$) it is well-known that one cannot proceed this way (see e.g. \cite{TaoVisanEtAl2007} and references therein); this will be an object of future research. Scattering is proven as in the cubic NLS case (proving boundedness of some Strichartz admissible norm) taking into consideration that the nonlocal term defines a Calder\'on-Zygmund operator.

Standing waves are found as critical points of the energy. We show that the latter is bounded below on any $L^2$-sphere
$$S(c)\defeq\big\{u\in H^1(\mathbb R^3;\C):\lnorm u\rnorm_2^2=c\big\}$$
of radius $\sqrt{c}$, so that we look for minimizers; the physical case corresponding to $c=1$. We show that the infimum $\gamma(c)$ of the energy on $S(c)$ is non-positive and that minimizers exist for all $c>c_b\defeq \max\{c>0:\gamma(c)=0\}$. On the other hand we show that for all $c<c_b$ minimizers do not exist; the case $c=c_b$ remains unclear. We also give upper and lower bounds for $c_b$.

The nonlocal term does not possess the full $SO(3)$ symmetries and is not monotone with respect to symmetric rearrangements. Thus the problem is lacking compactness, i.e., we are unable to apply Strauss' embedding theorem (\cite{Strauss1977}) and are forced to deploy a concentration-compactness argument. The reason for considering the problem on $S(c)$ and not on $S(1)$ lies in the fact that we cannot exclude dichotomy on $S(1)$ and we are led to study the subadditivity property of the mapping $c\mapsto \gamma(c)$. This we are also not able to prove directly but proceed as follows: we prove strict monotonicity by studying trajectories from $S(c_1)$ to $S(c_2)$ and then use a reflection argument similar to the one in \cite{Maris2016} together with some nonlocal identities taken from \cite{LopesMaris2008} to prove concavity. Our method is applicable for all $p\in (4,6]$, assuming that the highest order term is repulsive (defocusing).

Finally we would like to mention a number of works that have studied the dipolar GPE without the LHY-correction term (up to our knowledge this work is the first rigorous study of edGPE). That case differs from ours, since the dipolar term competes with the NLS term and yields an explicitly defined stable and unstable (blow-up of local solutions) parameter regime. A well-posedness theory and some dimension reduction results were first proven in \cite{Carles2008}. The threshold of global existence and finite time blow up in the focusing case was studied in \cite{MaCao2011}. Existence of solitary waves via a Weinstein-type scaling invariant functional was proven in \cite{AntonelliSparber2011}. Dimension reduction, ground states and dynamical properties of a condensate in anisotropic confinement were the subject of \cite{BaoEtAl2012}. A sharp blowup threshold was given in \cite{MaWang2013}, the case of a dipolar GPE system was studied in \cite{LiuEtAl2016}. Stability of standing waves and their symmetry and orbits was studied in \cite{CarlesHajaiej2015}. Standing waves in the unstable regime, scattering and stability were studied in \cite{BellazziniJeanjean2016}. More dimension reduction results, including cigar-shaped traps are found in \cite{BaoEtAl2017}. Standing waves that concentrate around local minima of the trapping potential were constructed in \cite{HeLuo2019}. Finally, a rigorous derivation from many-body quantum mechanics was done in \cite{Triay2018}.

\section{The extended dipolar Gross-Pitaevskii equation}
We study the equation in the following dimensionless form:
\begin{equation}\tag{edGPE}
 \label{edGPE} i\,\partial_t \psi=-\frac{1}{2}\,\Delta \psi+V_{ext}\,\psi+\lambda_1\,|\psi|^2\,\psi+\lambda_2\,(K*|\psi|^2)\,\psi+\lambda_3\,|\psi|^{p-2}\,\psi,\quad x\in\mathbb R^3,\ t>0,
\end{equation}
where $p\in(4,6]$, $K$ is a convolution kernel, $\lambda_1,\lambda_2,\lambda_3$ are given real constants. In particular, we consider
\begin{equation*}
K(x)=\frac{1-3\cos^2\theta(x)}{|x|^3},
\end{equation*}
where $\theta(x)$ is the angle between $x\in\mathbb R^3$ and a given (fixed) dipole axis $n\in\mathbb R^3$  with $|n| = 1$, i.e.,
\begin{equation*}
 \cos\theta(x) =\frac{x\cdot n}{|x|}.
\end{equation*}
We assume that the applied magnetic field is parallel to the $x_3$-axis, i.e., $n\defeq (0,0,1)$, so that
\begin{equation*}
 \label{K} K(x)= \frac{x_1^2+x_2^2-2x_3^2}{|x|^5}.
\end{equation*}
If we use the Fourier transform
\begin{equation*}
 \mathcal{F}(f)(\xi)=\widehat{f}(\xi)\defeq \int_{\R^3}f(x)e^{-ix\cdot \xi}\;dx
\end{equation*}
on $K$, we get
\begin{equation}\label{rangeKhat}
\widehat{K}(\xi)=\frac{4\pi}{3}\,\frac{2\xi_3^2-\xi_1^2-\xi_2^2}{|\xi|^2}\in\Big[-\frac{4}{3}\pi,\frac{8}{3}\pi\Big];
\end{equation}
see \cite[Lemma 2.3]{Carles2008}. When it comes to the trap, we consider two cases: either $V_{ext}=0$ (the ``self-bound'' case) or a potential well:
\begin{equation}\label{Vext}
\begin{aligned}
 &V_{ext}\in C^\infty(\R^3)\ \text{ and }\  \partial_\alpha V_{ext}\in L^\infty(\R^3)\ \text{ for all }\alpha\in \mathbb N^3,\ |\alpha|\geq 2,\\
 &\text{ such that }\ |x|\strongly+\infty \Rightarrow V_{ext}(x)\strongly+\infty.
\end{aligned}
\end{equation}
A typical trap is set with a harmonic potential:
\begin{equation*}
 V_{ext}(x)\defeq \frac{\omega_1^2}{\omega_3^2}\, x_1^2+\frac{\omega_2^2}{\omega_3^2}\,x_2^2+x_3^2,
\end{equation*}
where $\omega_1,\omega_2,\omega_3$ are the frequencies of the trap, in the $x_1,x_2,x_3$-directions respectively. Since the ``self-bound'' case seems to be the most technically challenging, we will present the proofs for the case $V_{ext}=0$. In the last section of the paper we will comment on and partially prove results for the case $V_{ext}\neq 0$.

Equation \eqref{edGPE} possesses a dynamically conserved energy functional $E$, which, for $u:\mathbb R\strongly \mathbb C$, is formally defined by
\begin{equation}
 \label{energy} E(u)\defeq \int_{\mathbb R^3}\Big\{\frac{1}{2}|\nabla u|^2+V_{ext}\,|u|^2+\frac{\lambda_1}{2}\,|u|^4+\frac{\lambda_2}{2}\,\big(K*|u|^2\big)\,|u|^2+\frac{2}{p}\,\lambda_3\,|u|^p\Big\}\;dx.
\end{equation}
With the help of Parseval's identity the latter becomes
\begin{equation*}
 \label{energy1} E(u)=\frac{1}{2}\|\nabla u\|_2^2+\|V_{ext}\, |u|^2\|_1+ \frac{1}{2}\frac{1}{(2\pi)^3}\,\int_{\R^3}\big(\lambda_1+\lambda_2\,\widehat{K}(\xi)\big)\,\big|\widehat{|u|^2}(\xi)\big|^2\;d\xi+\frac{2}{p}\lambda_3\,\|u\|_p^p.\\
\end{equation*}
For an arbitrary $c>0$, we look for ground states of \eqref{energy}, that is, for functions $u\in H^1(\mathbb R;\mathbb C)$ such that $\lnorm u\rnorm^2_{2}=c$, that are critical points of $E$ and study their qualitative properties. Note that a ground or excited state of $E$ corresponds to standing waves for \eqref{edGPE} through the Ansatz $\psi(x,t)=e^{-i\,\beta\,t}\,u(x)$; $\beta$ denotes the so-called \textit{chemical potential}. After making the standing wave Ansatz in \eqref{edGPE}, the problem reduces into finding a function $u:\R^3\strongly\C$ satisfying the side constraint $\lnorm u\rnorm^2_{2}=c$ and a number $\beta\in\R$ such that $(u,\beta)$ satisfies the \textit{Standing Wave extended dipolar Gross-Pitaevskii Equation}:
\begin{equation}\label{solution}\tag{SWedGPE}
-\frac{1}{2}\Delta u+V_{ext}\,u +\lambda_1|u|^2u+\lambda_2 (K*|u|^2)u+\lambda_3|u|^{p-2}u+\beta u=0.
\end{equation}

The rescaling we used (the same as in \cite{BellazziniJeanjean2016}) is such, that $c=1$ corresponds to the physical problem. The reason for studying the equation for a general $c>0$ is of technical nature and becomes apparent later in the paper (a brief explanation was given in the introduction).
\begin{definition}\label{definition of notations}
We will make extensive use of the following quantities:
\begin{align*}
A(u){}&\defeq \|\nabla u\|_2^2,\\[0.5em]
B(u){}&\defeq \frac{1}{(2\pi)^3}\,\int_{\R^3}\big(\lambda_1+\lambda_2\,\widehat{K}(\xi)\big)\,\big|\widehat{|u|^2}(\xi)\big|^2\;d\xi,\\[0.5em]
C(u){}&\defeq \lambda_3\|u\|_p^p,\\[0.5em]
Q(u){}& \defeq A(u)+\frac{3}{2}B(u)+\frac{3p-6}{p}C(u),\\
\Xi {}&\defeq \frac{1}{(2\pi)^3}\max\bigg\{\Big|\lambda_1-\lambda_2\frac{4\pi}{3}\Big|,\Big|\lambda_1+\lambda_2\frac{8\pi}{3}\Big|\bigg\}.
\end{align*}
\end{definition}
Throughout the paper we make the following banal assumption:
\begin{equation}
 \label{nondegeneracy} \tag{nondegeneracy}\lambda_1, \lambda_2\text{ do not vanish simultaneously, so that }\Xi\neq 0.
\end{equation}
\begin{remark}
 The ``virial functional'' $Q$ is closely related to the Pohozaev identity. It is defined as such, so that critical points will satisfy $Q(u)=0$ (Proposition \ref{betaneq0}).
\end{remark}
\begin{remark}
Due to \eqref{rangeKhat}, we have $|\lambda_1+\lambda_2 \widehat{K}(\xi)|\leq \Xi$ for all $\xi\in\R^3$. This is an optimal inequality, since it becomes an equality (with plus or minus sign) for $\lambda_1$, $\lambda_2$ having the same sign and $\widehat K(\xi)=-4\pi/3$ or $\widehat K(\xi)=8\pi/3$. We thus have the following optimal estimate
\begin{equation}
 \label{B_estimate} |B(u)|\leq \Xi\,\lnorm u\rnorm_4^4, \text{ for all }\lambda_1,\lambda_2\in \mathbb R \text{ and } u\in L^2(\mathbb R^3)\cap L^4(\mathbb R^3).
\end{equation}
\end{remark}
\begin{remark}
 With the above definitions (and with $V_{ext}=0$) the following identity holds:
  \begin{equation*}
   E(u)=\frac{1}{2}A(u)+\frac{1}{2}B(u)+\frac{2}{p}C(u)
  \end{equation*}
\end{remark}

\begin{remark}
 In the experiments, the dipole-dipole interaction can be tuned to be either attracting ($\lambda_2< 0$) or repulsive ($\lambda_2>0$); see \cite{GiovanazziEtAl2002}.
\end{remark}
We point out that the Laplacian $-\Delta:H^{s+2}(\R^3)\strongly H^s(\R^3)$ is well-defined for all $s\in\R$ (see for instance \cite[Theorem 3.41, p.71]{Abels2012}). On the other hand, the embedding $H^1(\R^3)\subset L^p(\mathbb R^3)$ for $p\in[2,6]$ and the continuity of the convolution operator with kernel $K$ in $L^p(\mathbb R^3)$ (\cite[Lemma 2.1]{Carles2008}) allows for the following (standard) definitions:
\begin{definition}
\begin{enumerate}
 \item Let $I\subseteq \mathbb R$ be an interval with $0\in I$ and $\psi_0\in  H^1(\R^3;\C)$. We call $\psi\in C\big(I; H^1(\R^3;\C)\big)\cap C^{1}\big(I; H^{-1}(\R^3;\C)\big)$ a strong solution to \eqref{edGPE} with initial value $\psi_0$, if \eqref{edGPE} is satisfied in $H^{-1}(\R^3;\C)$ for all $t\in\mathbb R$ and $\psi(0)=\psi_0$; in particular, if $I=\R$ we call the solution global.
 \item We call $(u,\beta)\in H^1(\R^3;\C)\times \R$ a solution to \eqref{solution}, if the latter is satisfied in $H^{-1}(\R^3;\C)$ (with no side constraints).
\end{enumerate}
\end{definition}

Solutions to \eqref{solution} will be constructed as critical points of the energy $E$ in the constraint set
\begin{equation}\label{definition of sets}
 S(c)\defeq \Big\{u\in \Sigma:\lnorm u\rnorm^2_{2}=c\Big\},
\end{equation}
where
\begin{equation}
 \Sigma\defeq \left\{
 \begin{aligned}
  &H^1(\R^3;\C), {}&&\text{ for } V_{ext}=0,\\
  &\big\{u\in H^1(\R^3;\C):V_{ext}\,|u|^2\in L^1(\R^3)\big\}, {}&& \text{ for }V_{ext}\text{ as in \eqref{Vext}}.
 \end{aligned}
 \right.
\end{equation}
The space $\Sigma$ is then a Banach space equipped with the norm
$$\|u\|_{\Sigma}=\|u\|_{H^1}+\|V_{ext}\,|u|^2\|^{\frac{1}{2}}_1;$$
see also \cite[Chapter 9.2]{Cazenave2003}. The following (non-)compactness result is then standard (see for instance \cite[Lemma 3.1]{Zhang2000}), since the trap is assumed to be coercive. It is the reason which makes the case $V_{ext}=0$ more challenging.
\begin{lemma}\label{compactness trapping}
The space $\Sigma$ is continuously embedded to $L^p(\R^3;\C)$ for all $p\in[2,6]$. If $V_{ext}\neq 0$, then the embedding is compact for $p\in[2,6)$. If $V_{ext}=0$ then $\Sigma$ is not compactly embedded in any $L^p$.
\end{lemma}
Moreover, for a more detailed exposition on the geometry of $S(c)$ as a Finsler manifold we refer to \cite{Berestycki1983} and references therein.

Finally, we define the infimum function $\gamma:[0,\infty)\strongly\mathbb R$ by
\begin{equation}
 \gamma(c)\defeq\inf_{u\in S(c)}E(u)\label{gamma_c}
\end{equation}
for $c>0$ and $\gamma(0)\defeq 0$, the (possibly infinite) number
\begin{equation}
 c_b\defeq\sup\{c>0:\gamma(c)=0\}\label{c_b}
\end{equation}
(depending only on $\lambda_1,\lambda_2,\lambda_3,p$).
We will also use optimal Gagliardo-Nirenberg inequalities in $\mathbb R^3$. We write them in the form
\begin{equation}
 \label{G-N}\lnorm u\rnorm_{2\sigma+2}^{2\sigma+2}\leq \mathrm C_{\sigma}^{2\sigma+2}\,\lnorm \nabla u\rnorm_{2}^{3\sigma}\;\lnorm u\rnorm_2^{2-\sigma}.
\end{equation}
The following result holds:
\begin{theorem}[\protect\cite{Weinstein1982}]\label{optimal_G-N}
 The optimal constant for \eqref{G-N} is given by $\mathrm C_\sigma=\left(\frac{\sigma+1}{\lnorm\psi\rnorm_2^{2\sigma}}\right)^{\frac{1}{s\sigma+2}}$, where $\psi$ is the ground state of the equation
 $$-\frac{3\sigma}{2}\Delta \psi+\left(1-\frac{\sigma}{2}\right)\psi+\psi^{2\sigma+1}=0.$$
\end{theorem}

\section{Main results}
As already noted, we present the results for the case $V_{ext}=0$; the case of a coercive trap will be treated in the last section. Our first results concern a well-posedness and small data scattering theory for \eqref{edGPE}, not covering the energy critical case (i.e., $p\neq 6$).
\begin{theorem}[Existence]\label{Theorem 2}
Let $V_{ext}=0$, $\lambda_3>0$ and $p\in(4,6)$. Then, for each $\psi_0\in H^1(\R^3;\C)$,  \eqref{edGPE} possesses a unique strong global solution $\psi$ with initial datum $\psi_0$. In particular,
\begin{enumerate}
\item $\psi\in L^\infty\big(\R;H^1(\R^3;\C)\big)$,
\item the particle number and energy conserve, i.e.,
\begin{align*}
\|\psi(t)\|_2^2=\|\psi_0\|_2^2\text{ and }E(\psi(t))=E(\psi_0) \text{ for all $t\in\R$, and}
\end{align*}
\item the initial value problem for \eqref{edGPE} is well-posed in $H^1(\R^3;\C)$.
\end{enumerate}
\end{theorem}
\begin{theorem}[Scattering]\label{scattering}
Let $V_{ext}=0$, $c>0,\,\lambda_3>0$ and $p\in(4,6)$. Then there exists some $\delta>0$ such that, for all $\psi_0\in H^1(\R^3;\C)$ with $\|\psi_0\|_{H^1}<\delta$, exist $\psi_{\pm}\in H^1(\R^3;\C)$ such that, for the unique global solution $\psi$ of \eqref{edGPE} with initial value $\psi_0$ (given by Theorem \ref{Theorem 2}), we have
\begin{align*}
\lim_{t\to\pm\infty}\big\|\psi(t)-e^{it\frac{\Delta}{2}}\psi_{\pm}\big\|_{H^1}=0,
\end{align*}
where $e^{it\frac{\Delta}{2}}$ denotes the unitary semigroup generated by $\frac{\Delta}{2}$.
\end{theorem}
The next theorem deals with existence and non-existence of standing waves; the energy critical case $p=6$ is included.
\begin{theorem}[Existence of standing waves]\label{Theorem 1.1}
Let $V_{ext}=0$, $c>0,\,\lambda_3>0$ and $p\in(4,6]$.
\begin{enumerate}
 \item Assume that $\lambda_1,\lambda_2$ satisfy either \eqref{case_a1} or \eqref{case_a2}, where
  \begin{align}
   &\lambda_2\geq 0\ \text{ and }\ \lambda_1-\frac{4\pi}{3}\lambda_2\geq 0,\label{case_a1}\\
   &\lambda_2< 0\ \text{ and }\ \lambda_1+\frac{8\pi}{3}\lambda_2\geq 0\label{case_a2}.
  \end{align}
 Then $E(u)>0$ for all $u\in S(c)$ and $\gamma(c)=0$, i.e., $E(u)$ possesses no minimizer on $S(c)$ for all $c\in(0,\infty)$.
 \item Assume that $\lambda_1,\lambda_2$ satisfy either \eqref{case_a3} or \eqref{case_a4}, where
 \begin{align}
  &\lambda_2\geq 0\ \text{ and }\ \lambda_1-\frac{4\pi}{3}\lambda_2< 0\label{case_a3}\\
  &\lambda_2< 0\ \text{ and }\ \lambda_1+\frac{8\pi}{3}\lambda_2< 0\label{case_a4}.
 \end{align}
 Then $c_b=\max\{c>0:\gamma(c)=0\}>0$. In particular, $E(u)>0$ for all $u\in S(c)$ and $E$ possesses no minimizer on $S(c)$ for all $c\in(0,c_b)$. On the other hand, $E$ possesses at least one minimizer $u$ on $S(c)$ for all $c\in(c_b,\infty)$.
\end{enumerate}
\end{theorem}

Next, we summarize some qualitative properties of minimizers. The symmetry assertions are obtained via the method of \cite{LopesMaris2008} consisting in a reflection argument and integral identities. The rest can be obtained using standard techniques.
\begin{proposition}[Qualitative properties of standing waves]\label{Theorem 1}
Let $V_{ext}=0$, $c>0$, $\lambda_3>0$, $p\in(4,6]$ and assume that $u\in S(c)$ is a minimizer of $E$ on $S(c)$. Then
\begin{enumerate}
\item Suppose that $p\neq 6$. Then
\begin{enumerate}
\item If $\lambda_2=0$, then $u$ is (up to translation) radially symmetric.
\item If $\lambda_2<0$, then $u$ is (up to translation) axially symmetric with respect to the $x_3$-axis.
\item If $\lambda_2>0$, there exists a minimizer $v$ such that it is (up to translations) symmetric with respect to the $(x_1,x_2)$-plane.
\end{enumerate}
\item If $p=6$, then
\begin{enumerate}
\item If $\lambda_2=0$, there exists a minimizer $v$ such that it is (up to translations) radially symmetric.
\item If $\lambda_2<0$, there exists a minimizer $v$ such that it is (up to translations) axially symmetric with respect to the $x_3$-axis.
\item If $\lambda_2>0$, there exists a minimizer $v$ such that it is (up to translations) symmetric with respect to the $(x_1,x_2)$-plane.
\end{enumerate}
\item The modulus $|u|$ is also a minimizer of $E$ on $S(c)$. Moreover, if $p\neq 6$, there exists some real number $\theta\in\R$ such that $u=e^{i\theta}|u|$ and $|u(x)|>0$ for all $x\in\R^3$.
\item If $p\neq 6$ and $(v,\beta)$ is a solution to \eqref{solution}, then $v$ is of class $W^{3,p}$ for all $p\in[2,\infty)$ and there exist constants $L,M>0$ such that
  \begin{equation*}
   e^{L|x|}\,\big(|v(x)|+|\nabla v(x)|\big)\leq M\ \text{ for all }\ x\in \R^3.
  \end{equation*}
\end{enumerate}
\end{proposition}

\section{Pohozaev, boundedness and positivity}
We start by proving some first properties of the model under consideration, and tools that will be needed in later analysis. Note that the Pohozaev identities cannot be extracted for the energy critical case by testing the equation with $x\cdot\nabla u$. However, we overcome this problem since we are dealing with minimizers. This section is devoted to the proof of the following proposition.
\begin{proposition}\label{betaneq0}
 Let $V_{ext}=0$, $c>0,\,\lambda_3>0$ and $p\in(4,6]$.
 \begin{enumerate}
  \item The energy $E$ is bounded below in $S(c)$. Moreover, $\gamma(c)\leq 0$.
  \item If $u\in S(c)$ is a minimizer of $E$ on $S(c)$, then there exists $\beta>0$ such that $(u,\beta)$ is a solution to \eqref{solution}.
  \item If $u\in S(c)$ is a minimizer of $E$ on $S(c)$, then the following Pohozaev identities hold:
  \begin{align}
   Q(u)&=A(u)+\frac{3}{2}B(u)+\frac{3p-6}{p}C(u)=0,\label{qu=0}\\
   \beta\|u\|_2^2&=-\frac{1}{4}B(u)+\frac{p-6}{2p}C(u).\label{bu+cu}
  \end{align}
 \end{enumerate}
\end{proposition}
\begin{proof}
 \textit{1.} Recall that
\begin{equation*}
E(u)=\frac{1}{2}A(u)+\frac{1}{2}B(u)+\frac{2}{p}C(u).
\end{equation*}
Suppose that $E(u)$ is unbounded below. Then there exists a sequence $\{u_n\}_{n\in\N}\subset S(c)$ with $E(u_n)\to -\infty$ as $n\to\infty$. It then follows directly that $A(u_n)$ and $C(u_n)$ are positive. Thus we must have $B(u_n)\to-\infty$ as $n\to \infty$. Since from \eqref{B_estimate} and the nondegeneracy assumption follows
$$\|u_n\|_4^4\geq \Xi^{-1}|B(u_n)|,$$
we obtain that $\|u_n\|_4\to \infty$ as $n\to\infty$. On the other hand, from H\"older and Gagliardo-Nirenberg inequalities we have
\begin{align*}
\|u\|_4&\leq \|u\|_p^{\frac{p}{2(p-2)}}\|u\|_2^{\frac{p-4}{2(p-2)}} =c^{\frac{p-4}{4(p-2)}}\lambda_3^{-\frac{1}{2(p-2)}}C(u)^{\frac{1}{2(p-2)}},\\
\|u\|_4^{8/3}&\leq \mathrm C_{1}^{8/3} \|\nabla u\|_2^{2}\|u\|_2^{2/3}=\mathrm C_{1}^{8/3} c^{1/3}A(u),
\end{align*}
where $\mathrm C_{1}$ is the corresponding Gagliardo-Nirenberg constant (see Theorem \ref{optimal_G-N}). Thus we obtain that
\begin{align}\label{infinity}
E(u_n)&= \frac{1}{2}A(u_n)+\frac{1}{2}B(u_n)+\frac{2}{p}C(u_n)\nonumber\\
&\geq \frac{1}{2\mathrm C_{1}^{8/3} c^{1/3}}\|u_n\|_4^{8/3}-\frac{\Xi}{2}\|u_n\|_4^4+\frac{2\lambda_3}{pc^{\frac{p-4}{2}}}\|u_n\|_4^{2(p-2)}\to\infty
\end{align}
as $\|u_n\|_4\to \infty$ (since $2(p-2)>4$), which is a contradiction. Therefore, $E$ is bounded below on $S(c)$.

For $t\in(0,\infty)$ we define the scaling
\begin{equation}\label{scaling_a1}
u^{t}(x)\defeq t^{3/2} u(tx).
\end{equation}
Transforming the corresponding integrals we obtain that
\begin{equation}\label{scaling_a2}
\begin{aligned}
\|u^t\|_2^2&=\|u\|_2^2,\\
A(u^t)&=t^{2}A(u),\\
B(u^t)&=t^3 B(u),\\
C(u^t)&= t^{\frac{3p}{2}-3}C(u).
\end{aligned}
\end{equation}
Thus it follows
$$E(u^t)=\frac{t^2}{2}A(u)+\frac{t^3}{2}B(u)+\frac{2t^{\frac{3p}{2}-3}}{p}C(u),$$
which implies that $E(u^t)$ converges to $0$ as $t$ shrinks to $0$. Thus we infer that $\gamma(c)\leq 0$.

\textit{2. and 3.} That $u$ solves \eqref{solution} with some chemical potential $\beta\in\C$ follows immediately from Lagrange multiplier theorem. Since $u$ is a minimizer, we obtain that the real function $t\mapsto E(u^t)$ is smooth in $(0,\infty)$ and has a minimum at $t=1$. Thus
\begin{align*}
0=\frac{d}{dt}E(u^t)\bigg|_{t=1}=A(u)+\frac{3}{2}B(u)+\frac{3p-6}{p}C(u)=Q(u),
\end{align*}
which shows \eqref{qu=0}.

Multiplying \eqref{solution} with $\bar{u}$ we obtain that
\begin{align}\label{eq1}
&\frac{1}{2}A(u)+B(u)+C(u)+\beta\|u\|^2_{2}=0.
\end{align}
Eliminating $A(u)$ from \eqref{qu=0} and \eqref{eq1} we obtain that
\begin{align*}
\beta\|u\|_2^2=-\frac{1}{4}B(u)+\frac{p-6}{2p}C(u),
\end{align*}
which shows \eqref{bu+cu} and $\beta\in\R$. It is left to show $\beta>0$. From \eqref{bu+cu} we obtain that
\begin{equation}\label{>0}
\begin{aligned}
\beta\|u\|_2^2&=-\frac{1}{4}B(u)+\frac{p-6}{2p}C(u)\\
&=-\frac{1}{2}\big(\frac{1}{2}A(u)+\frac{1}{2}B(u)+\frac{2}{p}C(u)\big)+\frac{1}{4}A(u)+\frac{p-4}{2p}C(u)\\
&=-\frac{1}{2}E(u)+\frac{1}{4}A(u)+\frac{p-4}{2p}C(u)\\
&=-\frac{1}{2} \gamma(c)+\frac{1}{4}A(u)+\frac{p-4}{2p}C(u)>0,
\end{aligned}
\end{equation}
since $\gamma(c)\leq 0$ due to \textit{1.} and $p>4$. This completes the proof.
\end{proof}
\begin{remark}\label{uniform boundedness}
Let $\{u_n\}_{n\in\N}\subset S(c)$ be a minimizing sequence, i.e., $E(u_n)=\gamma(c)+o(1)$. Then due to the uniform boundedness of $E(u_n)$ and of $B(u_n)$ (which is obtained in the above proof), we also obtain the uniform boundedness of $A(u_n)$ and $C(u_n)$.
\end{remark}

\section{Global well-posedness theory}
We will prove Theorem \ref{Theorem 1} using the so-called Kato's method. To that end we first show the existence of local solutions and then give uniform bounds in time.
\begin{proposition}\label{local existence}
Let $p\in (4,6)$ and $\lambda_3>0$. For each $\psi_0\in H^1(\R^3,\C)$ exist $T_{\min},T_{\max}\in(0,\infty]$ maximal, such that \eqref{edGPE} possesses a unique strong solution $\psi$ on the interval $(-T_{\min},T_{\max})$, with initial datum $\psi_0$. In particular, the particle number and energy conserve, i.e.,
\begin{align*}
\|\psi(t)\|_2^2=\|\psi_0\|_2^2\text{ and }E\big(\psi(t)\big)=E(\psi_0)\text{ for all }t\in(-T_{\min},T_{\max}).
\end{align*}
Moreover the initial value problem is locally well-posed in $H^1(\R^3;\C)$ (in the sense of \cite[Definition 3.1.5]{Cazenave2003}).
\end{proposition}
\begin{proof}
Define
$$g(u) \defeq (K*|u|^2) u$$
for $u\in H^1(\R^3;\C)$. In view of \cite[(4.4.21), (4.4.22)]{Cazenave2003} and if $V_{ext}=0$, Proposition \ref{local existence} will follow from \cite[Theorem 4.4.6]{Cazenave2003}, as long as we can prove that for every positive constant $M$, there exists some positive constant $\mathrm C(M)$, depending only on $M$, such that
\begin{align}
\|g(u)-g(v)\|_{\frac{4}{3}}&\leq \mathrm C(M)\|u-v\|_{4}\text{ and }\label{g_1}\\
\|g(u)\|_{W^{1,\frac{4}{3}}}&\leq \mathrm C(M)(1+\|u\|_{W^{1,4}})\label{g_2}
\end{align}
for all $u,v\in H^1(\R^3,\C)\cap W^{1,4}(\R^3,\C)$ with $\|u\|_{H^1},\|v\|_{H^1}\leq M$ (with $\rho=4$ and $r=4$ in \cite[(4.4.21), (4.4.22)]{Cazenave2003}). Concerning \eqref{g_1}, we obtain (using Plancherel's identity and the generalized H\"older inequality) that
\begin{align}\label{u-v_1}
\lnorm g(u)-g(v)\rnorm_{\frac{4}{3}}={}&\lnorm (K*|u|^2) u-(K*|v|^2) v\rnorm_{\frac{4}{3}}\nonumber\\
\leq {}&\lnorm \big(K*(|u|^2-|v|^2)\big)\,u\rnorm_{\frac{4}{3}}+\lnorm (K*|v|^2)\,(u- v)\rnorm_{\frac{4}{3}}\nonumber\\
\leq {}& \lnorm K*(|u|^2-|v|^2)\rnorm_{2}\,\lnorm u\rnorm_4+\lnorm K*|v|^2\rnorm_2\,\lnorm u-v\rnorm_4\nonumber\\
= {}&\frac{1}{(2\pi)^{3/2}}\Big(\lnorm \widehat{K}\,\mathcal{F}(|u|^2-|v|^2)\rnorm_{2}\,\lnorm u\rnorm_4+\lnorm \widehat{K}\,\mathcal{F}(|v|^2)\rnorm_2 \,\lnorm u-v\rnorm_4\Big)\nonumber\\
\leq {}& \frac{1}{(2\pi)^{3/2}}\frac{8\pi}{3}\Big(\lnorm \mathcal{F}(|u|^2-|v|^2)\rnorm_{2}\,\lnorm u\rnorm_4+\lnorm \mathcal{F}(|v|^2)\rnorm_2\,\lnorm  u-v\rnorm_4\Big)\nonumber\\
={}&\frac{8\pi}{3}\big(\lnorm |u|^2-|v|^2\rnorm_{2}\,\lnorm u\rnorm_4+\lnorm |v|^2\rnorm_2\,\lnorm u-v\rnorm_4\big)\nonumber\\
\leq {}&\frac{8\pi}{3}\big(\lnorm u-v\rnorm_{4}\,(\lnorm u\rnorm_4+\lnorm v\rnorm_4)\,\lnorm u\rnorm_4+\lnorm v\rnorm^2_4\,\lnorm u-v\rnorm_4\big)\nonumber\\
\leq{}& 4\pi\mathrm C^2\big(1+\lnorm u\rnorm^2_{H^1}+\lnorm v\rnorm^2_{H^1}\big)\,\lnorm u-v\rnorm_4,
\end{align}
where $\mathrm C$ is the constant appearing in the embedding $H^1\subset L^4$. Concerning \eqref{g_2}, taking $v=0$ in \eqref{u-v_1} we already see that $g(u)\in L^{\frac{4}{3}}(\R^3,\C)$. Hence, we only need to show that $\nabla g(u)\in L^{\frac{4}{3}}(\R^3,\C^3)$. We obtain that
\begin{align*}
\lnorm \nabla(K*|u|^2)\,u\rnorm_{\frac{4}{3}}={}&\lnorm \big(K*(\nabla(|u|^2)\big)\,u+(K*|u|^2)\,\nabla u\rnorm_{\frac{4}{3}}\nonumber\\
\leq {}&2\lnorm K*(u\nabla \bar{u})\rnorm_{2}\,\lnorm u\rnorm_4+\lnorm K*|u|^2\rnorm_2\,\lnorm \nabla u\rnorm_{L^4}\nonumber\\
\leq {}&\frac{8\pi}{3}(2\lnorm u\nabla \bar{u}\rnorm_2\lnorm u\rnorm_4+\lnorm |u|^2\rnorm_2\lnorm \nabla u\rnorm_4)\nonumber\\
\leq {}&\frac{8\pi}{3}(2\lnorm u\rnorm^2_4\lnorm \nabla u\rnorm_4+\lnorm u\rnorm^2_4\lnorm \nabla u\rnorm_4)\nonumber\\
={}&8\pi\lnorm u\rnorm^2_4\lnorm \nabla u\rnorm_4\leq 8\pi \mathrm C^2\,\big(1+\lnorm u\rnorm^2_{H^1}\big)\lnorm \nabla u\rnorm_{4}.
\end{align*}
Taking $\mathrm C(M) \defeq 8\pi \mathrm C^2(1+2M^2)$ finishes the proof.
\end{proof}
\begin{proof}[Proof of Theorem \ref{Theorem 2}]
To show the global well-posedness, one only needs to show that the local solution $\psi$ given by Proposition \ref{local existence} belongs to $L^\infty(\R;H^1(\R^3;\C))$, since the general results from \cite{Cazenave2003} that were used in the proof for local existence assert the blow-up alternative (i.e. the existence interval is maximal). In the same way we obtained \eqref{infinity}, we get that
\begin{equation}\label{boundedness u_4}
\begin{aligned}
E(\psi_0)=E(\psi(t))&= \frac{1}{2}A(\psi(t))+\frac{1}{2}B(\psi(t))+\frac{2}{p}C(\psi(t))\\
&\geq \frac{1}{2\mathrm C_1^{8/3} \|\psi(t)\|_2^{2/3}}\|\psi(t)\|_4^{8/3}-\frac{\Xi}{2}\|\psi(t)\|_4^4+\frac{2\lambda_3}{p\|\psi_0\|_2^{p-4}}\|\psi(t)\|_4^{2(p-2)}\\
&=\frac{1}{2\mathrm C_1^{8/3} \|\psi_0\|_2^{2/3}}\|\psi(t)\|_4^{8/3}-\frac{\Xi}{2}\|\psi(t)\|_4^4+\frac{2\lambda_3}{p\|\psi_0\|_2^{p-4}}\|\psi(t)\|_4^{2(p-2)},
\end{aligned}
\end{equation}
where $\mathrm C_1$ is the corresponding Gagliardo-Nirenberg constant (see Theorem \ref{optimal_G-N}). From \eqref{boundedness u_4}, one directly obtains that $\|\psi(t)\|_4$ is uniformly bounded in time by some positive constant $\mathrm C$ depending only on $\lambda_1,\lambda_2,\lambda_3,\psi_0$ and $p$. Now we also obtain that
\begin{equation}\label{boundedness nablapsi}
\begin{aligned}
\|\nabla \psi(t)\|_2^2&=2E(\psi(t))-B(\psi(t))-\frac{4}{p}C(\psi(t))\\
&\leq 2E(\psi_0)+\Xi\|\psi(t)\|_4^4\\
&\leq 2E(\psi_0)+\Xi \mathrm C^4,
\end{aligned}
\end{equation}
since $C(\psi(t))$ is positive for all $t\in\R$. Thus we obtain that $\|\nabla \psi(t)\|_2$ is uniformly bounded for all $t\in\R$. Together with particle conservation we obtain the result.
\end{proof}
\section{Existence of scattering states for small initial data}
We will use some shorthand definitions in order to keep the notation as simple as possible. First some function spaces: for $1\leq q,r,s\leq \infty$ and an interval $I\subseteq \R$ with $0\in I$ define
\begin{align*}
L^q_tL_x^r&\defeq L^q(0,t;L^r(\R^3;\C)),\\
L^q_t W_x^{1,r}&\defeq L^q(0,t;W^{1,r}(\R^3;\C)),\\
L^q_t H_x^{s}&\defeq L^q(0,t;H^{s}(\R^3;\C)),\\
L^q_I L_x^r&\defeq L^q(I;L^r(\R^3;\C)),\\
L^q_I W_x^{1,r}&\defeq L^q(I;W^{1,r}(\R^3;\C)),\\
L^q_I H_x^{s}&\defeq L^q(I;H^{s}(\R^3;\C)).
\end{align*}
We will also use the following notation:
\begin{enumerate}
\item We denote by $U(t):=e^{it\frac{\Delta}{2}}$ the unitary semigroup generated by $\frac{i\Delta}{2}$;
\item A pair $(q,r)$ is called an admissible pair, if $r\in[2,6]$ and
$$ \frac{2}{q}=3\,\Big(\frac{1}{2}-\frac{1}{r}\Big).$$
\item For a function $f:\R\to \C$, the function $\Psi_f$ is defined by
$$ \Psi_f(t)\defeq\int_0^tU(t-s)f(s)ds.$$
\end{enumerate}
We also recall the Strichartz estimates (see for instance \cite[Theorem 2.3.3, Corollary 2.3.6, Remark 2.3.8]{Cazenave2003}): For every admissible pair $(q,r)$ and $(q_1,r_1)$ there exist some positive constants $\mathrm C_{q,r}$, $\mathrm c_{q,r}$, $\mathrm C_{q,r,q_1,r_1}$ and $\mathrm c_{q,r,q_1,r_1}$ such that
\begin{equation}
\begin{aligned}
\|U(\cdot)\phi_1\|_{L_{I}^q L_x^r}&\leq \mathrm c_{q,r}\|\phi_1\|_{2},\\
\|U(\cdot)\phi_2\|_{L_{I}^q W_x^{1,r}}&\leq \mathrm C_{q,r}\|\phi_2\|_{H^1},\\
\|\Psi_{f_1}\|_{L^q_I L_x^r}&\leq \mathrm c_{q,r,q_1,r_1}\|f_1\|_{L^{q_1'}_I L_x^{r_1'}},\\
\|\Psi_{f_2}\|_{L^q_I W_x^{1,r}}&\leq \mathrm C_{q,r,q_1,r_1}\|f_2\|_{L^{q_1'}_I W_x^{1,r_1'}}
\end{aligned}
\end{equation} 
for all $\psi_1\in L^2(\R^3;\C)$, $\psi_2\in H^1(\R^3;\C)$, $f_1\in L^{q'}_I L_x^{r'}$ and $f_2\in L^{q_1'}_I W_x^{1,r_1'}$, where $q',\,r',\, q_1',r_1'$ are the corresponding conjugate exponents.
\begin{remark}
We will mainly use the following admissible pairs:
\begin{align*}
(\infty,2),\, \left(\frac{8}{3},4\right),\, \left(\frac{4p}{3(p-2)},p\right).
\end{align*}
\end{remark}
\begin{proof}[Proof of Theorem \ref{scattering}]
For convenience we will use $M$ for some positive constant which may vary from line to line, but depends only on $\lambda_1,\,\lambda_2,\,\lambda_3,\ p$ and Sobolev embedding constants, in various inequalities.

Recall Duhamel's formula
\begin{align}
\psi(t)=U(t)\psi_0-i\big(\lambda_1\Psi_{|\psi|^2\psi}(t)+\lambda_2\Psi_{(K*|\psi|^2)\psi}(t)+\lambda_3\Psi_{|\psi|^{p-2}\psi}(t)\big).
\end{align}
From \cite[Theorem 1.4]{BellazziniJeanjean2016} and the standard Strichartz estimates we already have:
\begin{equation}\label{strichartz1}
\begin{aligned}
\|\Psi_{|\psi|^2\psi}\|_{L_I^{\frac{4p}{3(p-2)}}W_x^{1,p}}&\leq \mathrm C_{\frac{4p}{3(p-2)},p,\frac{8}{3},4}\||\psi|^2\psi\|_{L_I^{\frac{8}{5}}W_x^{1,\frac{4}{3}}}\leq M(\|\psi_0\|_{H^1})\,\|\psi\|^{\frac{5}{3}}_{L_I^{\frac{8}{3}}W_x^{1,4}},\\
\|\Psi_{(K*|\psi|^2)\psi}\|_{L_I^{\frac{4p}{3(p-2)}}W_x^{1,p}}&\leq  \mathrm C_{\frac{4p}{3(p-2)},p,\frac{8}{3},4}\|(K*|\psi|^2)\psi\|_{L_I^{\frac{8}{5}}W_x^{1,\frac{4}{3}}}\leq M(\|\psi_0\|_{H^1})\,\|\psi\|^{\frac{5}{3}}_{L_I^{\frac{8}{3}}W_x^{1,4}}.
\end{aligned}
\end{equation}
Notice that \eqref{strichartz1} is based on the continuity of the convolution operator with kernel $K$ and on the boundedness of the $L^\infty_\R H_x^1$-norm of the solution $\psi$, which is ensured by Theorem \eqref{Theorem 2}; from its proof we obtain that $ M(\|\psi_0\|_{H^1})\strongly 0$ for $\|\psi_0\|_{H^1}\strongly 0$.

Now we derive the corresponding Strichartz estimates for the term $\Psi_{|\psi|^{p-2}\psi}$. Notice that the conjugate of the admissible pair $(\frac{4p}{3(p-2)},p)$ is $(\frac{4p}{p+6},\frac{p}{p-1})$. We obtain using H\"older's inequality that
\begin{equation}\label{strichartz2}
\begin{aligned}
\|\Psi_{|\psi|^{p-2}\psi}\|_{L_t^{\frac{4p}{3(p-2)}}L_x^{p}}\leq{}& \mathrm c_{\frac{4p}{3(p-2)},p,\frac{4p}{p+6},\frac{p}{p-1}}\||\psi|^{p-2}\psi\|_{L_t^{\frac{4p}{p+6}}L_x^{\frac{p}{p-1}}}\\
\leq{}& \mathrm c_{\frac{4p}{3(p-2)},p,\frac{4p}{p+6},\frac{p}{p-1}}\||\psi|^{p-2}\|_{L_t^{\frac{2p}{6-p}}L_x^{\frac{p}{p-2}}}\|\psi\|_{L_t^{\frac{4p}{3(p-2)}}L_x^p}\\
={}&\mathrm c_{\frac{4p}{3(p-2)},p,\frac{4p}{p+6},\frac{p}{p-1}}\|\psi\|^{p-2}_{L_t^{\frac{2p(p-2)}{6-p}}L_x^p}\|\psi\|_{L_t^{\frac{4p}{3(p-2)}}L_x^p}.
\end{aligned}
\end{equation}
On the other hand, define
\begin{align*}
\omega(s)\defeq\frac{2s(s-2)}{6-s}-\frac{4s}{3(s-2)}.
\end{align*}
Then
$$\omega(s)>0\Leftrightarrow s\in\Big(\frac{10}{3},6\Big).$$
Since $p\in(4,6)$, we obtain that $\omega(p)>0$, so that
\begin{equation}\label{strichartz3}
\begin{aligned}
\|\psi\|_{L_t^{\frac{2p(p-2)}{6-p}}L_x^p}\leq{}& \|\psi\|^{\frac{(6-p)\omega(p)}{2p(p-2)}}_{L_t^{\infty}L_x^p}\|\psi\|^{\frac{2(6-p)}{3(p-2)^2}}_{L_t^{\frac{4p}{3(p-2)}}L_x^p}\leq \|\psi\|^{\frac{(6-p)\omega(p)}{2p(p-2)}}_{L_t^{\infty}H_x^1}\|\psi\|^{\frac{2(6-p)}{3(p-2)^2}}_{L_t^{\frac{4p}{3(p-2)}}L_x^p}\\
\leq {}& M(\|\psi_0\|_{H^1})\,\|\psi\|^{\frac{2(6-p)}{3(p-2)^2}}_{L_t^{\frac{4p}{3(p-2)}}L_x^p},
\end{aligned}
\end{equation}
from the Sobolev embedding theorem and the fact that $\|\psi\|_{L_\R^{\infty}H_x^1}$ is bounded due to Theorem \ref{Theorem 2}. Then from \eqref{strichartz2} and \eqref{strichartz3} we obtain that 
\begin{align*}
\|\Psi_{|\psi|^{p-2}\psi}\|_{L_t^{\frac{4p}{3(p-2)}}L_x^{p}}\leq M(\|\psi_0\|_{H^1})\,\|\psi\|^{1+\frac{2(6-p)}{3(p-2)}}_{L_t^{\frac{4p}{3(p-2)}}L_x^p}= M(\|\psi_0\|_{H^1})\,\|\psi\|^{\frac{6+p}{3(p-2)}}_{L_t^{\frac{4p}{3(p-2)}}L_x^p}.
\end{align*}
Analogously we obtain that
\begin{align*}
\|\Psi_{\nabla(|\psi|^{p-2}\psi)}\|_{L_t^{\frac{4p}{3(p-2)}}L_x^{p}}\leq M(\|\psi_0\|_{H^1})\,\|\nabla \psi\|^{\frac{6+p}{3(p-2)}}_{L_t^{\frac{4p}{3(p-2)}}L_x^p}.
\end{align*}
The last two estimates imply that
\begin{align}
\|\Psi_{|\psi|^{p-2}\psi}\|_{L_t^{\frac{4p}{3(p-2)}}W_x^{1,p}}\leq M(\|\psi_0\|_{H^1})\,\|\psi\|^{\frac{6+p}{3(p-2)}}_{L_t^{\frac{4p}{3(p-2)}}W_x^{1,p}}.
\end{align}
Next, we estimate $\|\psi\|_{L_t^{\frac{8}{3}}W_x^{1,4}}$ by $\|\psi\|_{L_t^{\frac{4p}{3(p-2)}}W_x^{1,p}}$. Using H\"older's inequality and Sobolev embedding again we obtain that
\begin{align*}
\|\psi\|^{\frac{8}{3}}_{W_x^{1,4}}\leq \|\psi\|^{\frac{4p}{3(p-2)}}_{W_x^{1,p}}\|\psi\|^{\frac{4(p-4)}{3(p-2)}}_{H_x^1}\leq M(\|\psi_0\|_{H^1})\,\|\psi\|^{\frac{4p}{3(p-2)}}_{W_x^{1,p}}.
\end{align*}
Thus
\begin{align}
\|\psi\|_{L_t^{\frac{8}{3}}W_x^{1,4}}\leq M(\|\psi_0\|_{H^1})\,\|\psi\|^{\frac{p}{2(p-2)}}_{L_t^{\frac{4p}{3(p-2)}}W_x^{1,p}}\Rightarrow\|\psi\|^{\frac{5}{3}}_{L_t^{\frac{8}{3}}W_x^{1,4}}\leq M(\|\psi_0\|_{H^1})\,\|\psi\|^{\frac{5p}{6(p-2)}}_{L_t^{\frac{4p}{3(p-2)}}W_x^{1,p}}.
\end{align}
To sum up, we obtain from Duhamel's formula and the Strichartz estimates that
\begin{equation}\label{strichartz4}
\begin{aligned}
\|\psi\|_{L_t^{\frac{4p}{3(p-2)}}W_x^{1,p}}&\leq\|U(\cdot)\psi_0\|_{L_t^{\frac{4p}{3(p-2)}}W_x^{1,p}}+M\Big(\|\Psi_{|\psi|^2\psi}\|_{L_t^{\frac{4p}{3(p-2)}}W_x^{1,p}}\\
&\ \ +\|\Psi_{(K*|\psi|^2)\psi}\|_{L_t^{\frac{4p}{3(p-2)}}W_x^{1,p}}+\|\Psi_{|\psi|^{p-2}\psi}\|_{L_t^{\frac{4p}{3(p-2)}}W_x^{1,p}}\Big)\\
&\leq \mathrm C_{\frac{4p}{3(p-2)},p}\|\psi_0\|_{H^1}+M(\|\psi_0\|_{H^1})\,\Big(\|\psi\|^{\frac{5}{3}}_{L_t^{\frac{8}{3}}W_x^{1,4}}+\|\psi\|^{\frac{6+p}{3(p-2)}}_{L_t^{\frac{4p}{3(p-2)}}W_x^{1,p}}\Big)\\
&\leq \mathrm C_{\frac{4p}{3(p-2)},p}\|\psi_0\|_{H^1}+M(\|\psi_0\|_{H^1})\,\Big(\|\psi\|^{\frac{5p}{6(p-2)}}_{L_t^{\frac{4p}{3(p-2)}}W_x^{1,p}}+\|\psi\|^{\frac{6+p}{3(p-2)}}_{L_t^{\frac{4p}{3(p-2)}}W_x^{1,p}}\Big).
\end{aligned}
\end{equation}
Now define the function $f:[0,\infty)\to\R$ by
$$f(y)\defeq y-\mathrm C_{\frac{4p}{3(p-2)},p}\|\psi_0\|_{H^1}-M(\|\psi_0\|_{H^1})\,y^{\frac{5p}{6(p-2)}}-M(\|\psi_0\|_{H^1})\,y^{\frac{6+p}{3(p-2)}},$$
where $ M(\|\psi_0\|_{H^1})\strongly 0$ for $\|\psi_0\|_{H^1}\strongly 0$. Notice that $\frac{5p}{6(p-2)}>1$ and $\frac{6+p}{3(p-2)}>1$ for $p\in(4,6)$. Thus choosing $\|\psi_0\|_{H^1}$ sufficiently small, say $\|\psi_0\|<\delta$ for some $\delta>0$, there exist some positive constants $a,b$ with $0<a<b<\infty$ such that
$$\{y\in(0,\infty):f(y)<0\}\subseteq (0,a)\cup(b,\infty).$$
Since \eqref{strichartz4} is valid for all $t\in(0,\infty)$, $\|\psi\|_{L_t^{\frac{4p}{3(p-2)}}W_x^{1,p}}$ converges to zero as $t$ shrinks to zero and the mapping $t\mapsto \|\psi\|_{L_t^{\frac{4p}{3(p-2)}}W_x^{1,p}}$ is continuous, we obtain that
\begin{equation*}
\|\psi\|_{L_t^{\frac{4p}{3(p-2)}}W_x^{1,p}}\leq a,\ \text{ for all} \ t\in(0,\infty),
\end{equation*}
which implies that
\begin{alignat}{2}
&\|\psi\|_{L_\R^{\frac{4p}{3(p-2)}}W_x^{1,p}}\leq a.
\end{alignat}
Defining $g(\psi)\defeq \lambda_1|\psi|^2\psi+\lambda_2 (K*|\psi|^2)\psi+\lambda_3 |\psi|^{p-2}\psi$ and $v(t)\defeq U(-t)\psi(t)$, we obtain for $t,\tau\in(0,\infty)$ that
\begin{align*}
\|v(t)-v(\tau)\|_{H^1}\leq {}&M(\|\psi_0\|_{H^1})\,\|g(\psi)\|_{L_{(t,\tau)}^{\frac{4p}{p+6}}W_x^{1,\frac{p}{p-1}}}\notag\\
\leq{}& M(\|\psi_0\|_{H^1})\,\|\psi\|_{L_{(t,\tau)}^{\frac{4p}{3(p-2)}}W_x^{1,p}}^{\frac{5p}{6(p-2)}}+M\|\psi\|_{L_{(t,\tau)}^{\frac{4p}{3(p-2)}}W_x^{1,p}}^{\frac{6+p}{3(p-2)}}\strongly 0
\end{align*}
as $t,\tau\to\infty$, by dominated convergence. Therefore, $\{v(t)\}_{t\geq 0}$ is a Cauchy net in $H^1(\R^3;\C)$. We denote its $H^1$-limit by $\psi_+$. Then
\begin{align}
\lim_{t\to\infty}\|\psi(t)-\Psi(t)\psi_{+}\|_{H^1}=\lim_{t\to\infty}\|U(-t)\psi(t)-\psi_{+}\|_{H^1}=0.
\end{align}
This shows the existence of a scattering state $\psi_+$. Analogously we show a scattering state $\psi_-$ for $t\to-\infty$. This completes the proof.
\end{proof}
\section{Upper and lower estimates for the critical mass}
In this section we give some quantitative estimates on the regime where no standing waves exist. We also want to point out that such estimates imply that $c_b$ (which is given by \eqref{c_b}) is indeed finite (in the suitable parameter regime \eqref{case_a3}-\eqref{case_a4}). In the proof of Theorem \ref{Theorem 1} we will show that the regimes that we find in this section can be optimized to $(0,c_b)$ and $(c_b,\infty)$. We are however unable to provide with a closed form characterization for $c_b$.

We first construct $c_a$ such that $c_a\leq c_b$.
\begin{lemma}\label{Nonexistence of minimizers}
Let $p\in(4,6]$ and $(\lambda_1,\lambda_2,\lambda_3)\in\R^2\times (0,\infty)$. Then there exists some $c_a>0$, depending on $\lambda_1,\lambda_2,\lambda_3,p$, such that for all $c\in(0,c_a)$ we have $E(u)>0$ for all $u\in S(c)$ and $\gamma(c)=0$, i.e., $E(u)$ possesses no minimizer on $S(c)$. Moreover, if $(\lambda_1,\lambda_2)$ satisfies \eqref{case_a1} or \eqref{case_a2}, then $c_a=\infty$.
\end{lemma}
\begin{proof}
If $B(u)$ is nonnegative, then we have already $E(u)>0$, so let us assume that $B(u)<0$. We then discuss two cases: $\|u\|_4\geq 1$ and $\|u\|_4<1$. First assume that $\|u\|_4\geq 1$. We obtain from H\"older's inequality that
\begin{align*}
\|u\|_4^4\leq \|u\|_4^{2(p-2)}\leq\|u\|^p_p\|u\|_2^{p-4}=c^{\frac{p-4}{2}}\,\|u\|_p^p.
\end{align*}
Also recall estimate \eqref{B_estimate}; we obtain that
\begin{align*}
E(u)=&\frac{1}{2}A(u)+\frac{1}{2}B(u)+\frac{2}{p}C(u)\\
\geq &\frac{1}{2}A(u)+\frac{2\lambda_3}{p}\|u\|_p^p-\frac{\Xi}{2}\|u\|_4^4\\
\geq &\frac{1}{2}A(u)+\Big(\frac{2\lambda_3c^{-\frac{p-4}{2}}}{p}-\frac{\Xi}{2}\Big)\|u\|_4^4>0
\end{align*}
for $c< \left(\frac{16\lambda_3^2}{p^2\Xi^2}\right)^{\frac{1}{p-4}}$. Now let $\|u\|_4<1$. Recall the Gagliardo-Nirenberg inequality
$$\|u\|^4_4\leq \mathrm C_1^4\,\|\nabla u\|_2^{3}\,\|u\|_2=\mathrm C_1^4\, c^{1/2}\, A(u)^{3/2},$$
(where $\mathrm C_1$ is an optimal constant given by Theorem \ref{optimal_G-N}). Suppose that $E(u)\leq 0$. We obtain that
\begin{align*}
\frac{1}{2}A(u)-\frac{\Xi \mathrm C_{1}^4c^{1/2}}{2}A(u)^{3/2}\leq \frac{1}{2}A(u)+\frac{1}{2}B(u)=E(u)-\frac{2}{p}C(u)<0,
\end{align*}
which implies that
$$A(u)>\Xi^{-2}\mathrm C_{1}^{-8}c^{-1}.$$
On the other hand,
\begin{align*}
0\geq E(u)=\frac{1}{2}A(u)+\frac{1}{2}B(u)+\frac{2}{p}C(u)>\frac{1}{2}A(u)-\frac{\Xi}{2}\|u\|_4^4\geq \frac{1}{2}A(u)-\frac{\Xi}{2},
\end{align*}
since $\|u\|_4\leq 1$. Thus letting
$$ \frac{1}{2}\Xi^{-2}\mathrm C_{1}^{-8}c^{-1}\geq \frac{\Xi}{2}\Leftrightarrow c\leq \Xi^{-3}\mathrm C_{1}^{-8},$$
we obtain the contradiction $0>0$. Then
$$c_a\defeq \min\left\{\left(\frac{16\lambda_3^2}{p^2\Xi^2}\right)^{\frac{1}{p-4}},\Xi^{-3}\mathrm C_{1}^{-8}\right\},$$
where $\mathrm C_1$ is given by Theorem \ref{optimal_G-N}, satisfies the assertion of the lemma. In particular, if \eqref{case_a1} or \eqref{case_a2} is satisfied, $B(u)$ is always nonnegative, thus $c_a=\infty$.
\end{proof}

In what follows, we construct $c_c$, depending only on $\lambda_1,\lambda_2,\lambda_3,p$, such that $\gamma(c_c)<0$. The construction is implicit and we present it for the cases $p=5,6$. From the monotonicity of the infimum function (which will be proven in the next section) we obtain that $c_b\leq c_c$.

 The main idea is to calculate the energy of a Gaussian and tune the parameters such that it becomes negative; all calculations were made with Mathematica 10.2. To that end, for $\sigma,\tau,c>0$ define
 $$u_{\sigma,\tau,c}(x)\defeq \sqrt{\frac{8 c}{\pi ^{3/2} \sigma ^2 \tau }} \exp \left(-2 \left(\frac{x_1^2+x_2^2}{\sigma ^2}+\frac{x_3^2}{\tau ^2}\right)\right)$$
 and note that $u_{\sigma,\tau,c}\in S(c)$. One obtains that
 \begin{align*}
  E(u_{\sigma,\tau,c})={}&\tilde A(\sigma,\tau)\,c+\tilde B_{\lambda_1,\lambda_2}(\sigma,\tau)\,c^2+\tilde C_{\lambda_3,p}(\sigma,\tau)\,c^{p/2},
 \end{align*}
 where
 $$\tilde A(\sigma,\tau)\defeq \frac{2}{\sigma ^2}+\frac{1}{\tau ^2},\quad \tilde C_{\lambda_3,p}(\sigma,\tau)\defeq\frac{\lambda_3  \,2^{1+3 (p-1)/2}   }{\pi ^{3 (p-2)/4}\,p^{5/2}\,\sigma ^{p-2}\, \tau^{(p-2)/2}}$$
 and
 \begin{equation*}
  \tilde B_{\lambda_1,\lambda_2}(\sigma,\tau)\defeq \left\{
  \begin{aligned}
   \frac{\sqrt{2}}{\pi ^{3/2}\left( \sigma^2 -\tau ^2\right)}\left(\frac{1}{\tau }\Big(\lambda_1 +\frac{8}{3} \pi  \lambda_2\Big)-\frac{\tau}{\sigma^2}\Big(\lambda_1-\frac{4}{3} \pi  \lambda_2 \Big)-\frac{4 \pi  \lambda_2  \coth ^{-1}\left(\frac{\tau }{\sqrt{\tau ^2-\sigma ^2}}\right)}{\sqrt{\tau ^2-\sigma ^2}}\right)\\
   \text{ if }\tau>\sigma,\\[2em]
   \frac{\sqrt{2}}{\pi ^{3/2}\left( \sigma^2 -\tau ^2\right)}\left(\frac{1}{\tau }\Big(\lambda_1 +\frac{8}{3} \pi  \lambda_2\Big)-\frac{\tau}{\sigma^2}\Big(\lambda_1-\frac{4}{3} \pi  \lambda_2 \Big)-\frac{4 \pi  \lambda_2  \cot^{-1}\left(\frac{\tau }{\sqrt{\sigma ^2-\tau ^2}}\right)}{\sqrt{\sigma ^2-\tau ^2}}\right)\\
   \text{ if }\tau<\sigma.\\
  \end{aligned}
  \right.
 \end{equation*}

 We first treat the case $p=5$. The inequality $a_1 c + a_2 c^2 + a_3 c^{5/2} < 0$ with positive coefficients $a_1$ and $a_3$ is satisfied for $c\in (P_2,P_3)$, where $P_i$ denotes the $i$th root of the polynomial
 $$P(s)\defeq a_3^2\,s^3- a_2^2\,s^2-2a_1 a_2\, s -a_1^2,$$
 provided $a_1<-\frac{4 a_2^3}{27 a_3^2}$. Defining $a_1,a_2,a_3$ by the expression for the energy above and setting $\sigma=\sqrt{\tau}$, we find that
 $$\lim_{\tau\rightarrow\infty}\left( -\frac{4 \tilde B_{\lambda_1,\lambda_2}(\sqrt{\tau},\tau)^3}{27 \tilde C_{\lambda_3,p}(\sqrt{\tau},\tau)^2}\right)=-\frac{3125 (3 \lambda_1-4 \pi  \lambda_2)^3}{110592 \sqrt{2} \lambda_3^2}>0$$
 for $\lambda_1,\lambda_2$ satisfying \eqref{case_a3}. Since $\tilde A(\sqrt{\tau},\tau)$ becomes arbitrarily small for $\tau$ arbitrarily large, there exists $\tau$ sufficiently large such that $E(u_{\sqrt{\tau},\tau,c})<0$.

 For $\lambda_1,\lambda_2$ satisfying \eqref{case_a4}, we rescale like $\tau=\sqrt{\sigma}$ to find
 $$\lim_{\sigma\rightarrow\infty}\left( -\frac{4 \tilde B_{\lambda_1,\lambda_2}(\sigma,\sqrt{\sigma})^3}{27 \tilde C_{\lambda_3,p}(\sigma,\sqrt{\sigma})^2}\right)=-\frac{3125 (3 \lambda_1+8 \pi  \lambda_2 )^3}{110592 \sqrt{2} \lambda_3 ^2}>0,$$
 so that $E(u_{\sigma,\sqrt{\sigma},c})<0$ for $\sigma$ sufficiently large.

 Concerning the case $p=6$, solving the inequality $a_1 c + a_2 c^2 + a_3 c^{5/2} < 0$ for positive $c$, is equivalent to having  two distinct roots for the binomial $P(c)\defeq a_1 + a_2 c + a_3 c^{2}$. This happens for $a_2<0$ and $a_2^2>4a_1\,a_3$. Calculating
 $$\lim_{\sigma\nearrow \tau}\big(\tilde B_{\lambda_1,\lambda_2}(\sigma,\tau)^2-4\tilde A(\sigma,\tau)\,\tilde C_{\lambda_3,p}(\sigma,\tau)\big)=\frac{2 \lambda_1^2}{\pi ^3 \tau ^6}-\frac{1536 \lambda_3 }{25 \sqrt{5} \pi ^{9/4}\,\tau^{13/2}},$$
 we see that the positive term dominates the negative one for $\tau$ sufficiently small. Thus, there exist $\tau$ small enough and $\sigma<\tau$ sufficiently close to $\tau$ such that
 $$\tilde B_{\lambda_1,\lambda_2}(\sigma,\tau)^2-4\tilde A(\sigma,\tau)\,\tilde C_{\lambda_3,p}(\sigma,\tau)>0$$
 and thus $E(u_{\sigma,\tau,c})<0$.

\section{Monotonicity and concavity of the infimum function}
The following construction is originally proved in \cite{BellazziniJeanjean2016}.
\begin{lemma}\label{bu<0}
Let \eqref{case_a3} or \eqref{case_a4} be satisfied. Then for each $c>0$ there exists some $u\in S(c)$ such that $B(u)<0$.
\end{lemma}
\begin{proof}
 We only consider the case \eqref{case_a3}, namely $\lambda_2\geq 0$ and $\lambda_1-\frac{4\pi}{3}\lambda_2<0$, the case \eqref{case_a4} can be dealt similarly. We define the following scaling
\begin{align*}
u_t(x) \defeq t^{\frac{5}{4}}u(tx_1,t x_2,\sqrt{t}x_3)
\end{align*}
for $t>0$. Then
$$B(u_t)=\frac{t^{\frac{5}{2}}}{(2\pi)^3}\int_{\R^3}\Big(\lambda_1+\frac{4\pi}{3}\lambda_2\frac{2t\xi_3^2-t^2\xi_1^2-t^2\xi_2^2}{t^2\xi_1^2+t^2\xi_2^2+t^2\xi_3^2}\Big)\,\big|\widehat{|u(\xi)|^2}\big|^2\;d\xi. $$
Letting $t\to\infty$ we obtain that
\begin{align*}
\lim_{t\to\infty}\lambda_1+\frac{4\pi}{3}\lambda_2\frac{2t\xi_3^2-t^2\xi_1^2-t^2\xi_2^2}{t^2\xi_1^2+t^2\xi_2^2+t^2\xi_3^2}=\lambda_1-\frac{4\pi}{3}\lambda_2<0.
\end{align*}
Then using the dominated convergence theorem we obtain that there exists some sufficiently large $t$ such that $B(u_t)<0$. This completes the proof.
\end{proof}

We now prove the monotonicity of the infimum function. In order to do so, we need to work with a rescaling that leaves the highest order term invariant. We are then able to study how the energy varies with respect to the $L^2$-norm, in order to obtain the result.
\begin{lemma}\label{strict decreasing}
Let $\lambda_3>0$ and $p\in(4,6]$. Then $c\mapsto \gamma(c)$ is nonincreasing on $(0,\infty)$. In particular, there exists some $c_0\in(0,\infty)$ such that $\gamma(c_0)<0$. Moreover, for each $c_0\in(0,\infty)$ satisfying $\gamma(c_0)<0$, the function $c\mapsto \gamma(c)$ is strictly decreasing on $[c_0,\infty)$.
\end{lemma}
\begin{proof}
We first define the following scaling
\begin{align}\label{scaling_b1}
{}^{t}u(x) \defeq t^{-3/p}u(t^{-1}x).
\end{align}
We obtain that
\begin{equation}\label{scaling_b2}
\begin{aligned}
\|{}^{t}u\|_2^2&=t^{3-\frac{6}{p}}\|u\|_2^2,\\
A({}^{t}u)&=t^{1-\frac{6}{p}}A(u),\\
B({}^{t}u)&=t^{3-\frac{12}{p}}B(u),\\
C({}^{t}u)&=C(u),
\end{aligned}
\end{equation}
so that it holds
$$ E({}^{t}u)=\frac{1}{2}t^{1-\frac{6}{p}}A(u)+\frac{1}{2}t^{3-\frac{12}{p}}B(u)+\frac{2}{p}C(u).$$
Due to Lemma \ref{bu<0}, for a given $c>0$ we can find some $u\in S(c)$ with $B(u)<0$. Since
\begin{align*}
1\leq \frac{6}{p}\ \text { and }\ 3>\frac{12}{p},
\end{align*}
we obtain that $E({}^{t}u)\rightarrow-\infty$ for $t\rightarrow \infty$ and $t\mapsto E({}^{t}u)$ is strictly decreasing on $t\in[1,\infty)$. This shows the existence of some $c_0\in(0,\infty)$ with $\gamma(c_0)<0$.

Next we show that $c\mapsto \gamma(c)$ is nonincreasing on $(0,\infty)$. Since $\gamma(c)=0$ for all $c\in(0,\infty)$ if $\lambda_1$ and $\lambda_2$ satisfy \eqref{case_a1} or \eqref{case_a2}, we need only consider the cases \eqref{case_a3} and \eqref{case_a4}. Let $0<c_1<c_2<\infty$ be given. If $\gamma(c_1)=0$, then nothing has to be shown since $\gamma(c)\leq 0$ due to Proposition \refpart{betaneq0}{1} We thus assume that $\gamma(c_1)<0$. Let $\{u_n\}_{n\in\N}\subset S(c_1)$ be a minimizing sequence, i.e., $E(u_n)=\gamma(c_1)+o(1)$. Up to a subsequence we can thus assume that $B(u_n)<0$ for all $n\in \N$. Letting $t=(\frac{c_2}{c_1})^{\frac{p}{3p-6}}>1$ (since $\frac{p}{3p-6}\Leftrightarrow p>3$), we obtain that $\|{}^{t}u_n\|_2^2=c_2$. Due to the strict deceasing monotonicity of $t\mapsto E({}^{t}u_n)$ in $[1,\infty)$ we obtain that
$$ \gamma(c_1)+o(1)=E(u_n)\geq E({}^{t}u_n)\geq \gamma(c_2),$$
which implies that
$$\gamma(c_1)\geq \gamma(c_2) $$
and this completes the proof of nonincreasing monotonicity of $c\mapsto \gamma(c)$ on $(0,\infty)$.

Now let $\gamma(c_0)<0$. Supposing that $c\mapsto \gamma(c)$ is not strictly decreasing on $[c_0,\infty)$, we can find two points $c_1$ and $c_2$ with $c_0\leq c_1<c_2<\infty$ and $\gamma(c_1)=\gamma(c_2)$. Then from the nonincreasing monotonicity of $c\mapsto\gamma(c)$ we obtain that \begin{align}\label{contradiction}
\gamma(c)=\gamma(c_1)=\gamma(c_2)<0
\end{align}
for all $c\in[c_1,c_2]$. Again we let $\{u_n\}_{n\in\N}\subset S(c_1)$ be a minimizing sequence in $S(c_1)$. We obtain from the scaling \eqref{scaling_b2} that
\begin{align}\label{difference of aun and bun}
E({}^{t}u_n)=E(u_n)+\frac{1}{2}(t^{1-\frac{6}{p}}-1)A(u_n)+\frac{1}{2}(t^{3-\frac{12}{p}}-1)B(u_n).
\end{align}
We again assume that $B(u_n)<0$ for all $n\in\N$. Moreover, it holds that
$$\sigma \defeq \lim_{n\to\infty} B(u_n)\in(-\infty,0),$$
since: (i) if $\liminf_{n\to\infty} B(u_n)=-\infty$, then we arrive to a contradiction as in the proof of Proposition \refpart{betaneq0}{1}; (ii) if $\liminf_{n\to\infty} B(u_n)\geq 0$, then we obtain
\begin{align}\label{minimizing equal zero}
0>\gamma(c_1)=\lim_{n\to\infty}E(u_n)=\frac{1}{2} A(u_n)+\frac{1}{2}B(u_n)+\frac{3}{5}C(u_n)\geq 0,
\end{align}
a contradiction. Since $E({}^{t}u_n)\strongly -\infty$ for $t\strongly+\infty$ and $t\mapsto E({}^{t}u_n)$ is strictly decreasing on $t\in[1,\infty)$, we obtain the existence of some $t_n\in[1,\infty)$ with $E({}^{t_n}u_n)=\gamma(c_1)$. Hence
\begin{align}\label{aun+bun}
o(1)&=\gamma(c_1)-E(u_n)\nonumber\\
&=\frac{1}{2}(t_n^{1-\frac{6}{p}}-1)A(u_n)+\frac{1}{2}(t_n^{3-\frac{12}{p}}-1)B(u_n),\nonumber\\
&\leq \frac{1}{2}(t_n^{3-\frac{12}{p}}-1)B(u_n)\eqdef l_n<0.
\end{align}
We obtain that $l_n=o(1)$. But since $B(u_n)\neq o(1)$ we infer that $t_n\strongly 1$ for $n\rightarrow \infty$. Now fix $c\in(c_1,c_2)$. If there is some $t_n=1$, then $u_n$ is a minimizer, and we obtain for $t=(c/c_1)^{\frac{p}{3p-6}}$ that $\|{}^{t}u_n\|_2^2=c$ and
$$\gamma(c)\leq E({}^{t}u_n)<E(u_n)=\gamma(c_1),$$
which contradicts \eqref{contradiction}. Thus $t_n>1$ for all $n\in\N$. Since $\{t_n\}_{n\in\mathbb N}$ converges to $1$, we can find some sufficiently large $n\in\N$ and sufficiently small $\varepsilon_n>0$ such that $\hat{c}=(t_n+\varepsilon_n)^{\frac{3p-6}{p}}c_1\in(c_1,c_2)$ and
\begin{align*}
\|{}^{t_n+\varepsilon_n}u_n\|_2^2&=\hat{c}\in(c_1,c_2)\text{ and } \gamma(\hat{c})\leq E({}^{t_n+\varepsilon_n}u_n)<E({}^{t_n}u_n)=\gamma(c),
\end{align*}
which contradicts \eqref{contradiction} again. This completes the proof.
\end{proof}

The concavity of the infimum function is proven using symmetry arguments like the ones given in \cite{Maris2016}, as well as an integral identity which was crucial for proving symmetry for minimizers in \cite{LopesMaris2008}.
\begin{lemma}\label{maris}
Let $\lambda_3>0$ and $p\in(4,6]$. Then
\begin{enumerate}
\item The function $\gamma:[0,\infty)\strongly\mathbb R$ is concave and thus continuous.
\item If there exist $c_1,c_2\in(0,\infty)$ with $\gamma(c_1+c_2)=\gamma(c_1)+\gamma(c_2)$, then $\gamma$ is linear in $[0,c_1+c_2]$.
\end{enumerate}
\end{lemma}
\begin{proof}
We follow the lines of \cite[Theorem 1.1]{Maris2016}. First, we show that
\begin{align}\label{concave 1}
\gamma\Big(\frac{1}{2}(c_1+c_2)\Big)\geq \frac{1}{2}\big(\gamma(c_1)+\gamma(c_2)\big)
\end{align}
for all $c_1,c_2\in[0,\infty)$: For a function $u\in H^1(\R^3,\C)$, we define the following extensions of $u$ with respect to the hyperplane $\{x\in\R^3:x_1=t\}$:
\begin{align*}
u^1_{1,t}(x)\defeq\left\{
             \begin{array}{ll}
                 u(x) &\text{ if }\,x_1\leq t,  \\
             u (2t-x_1,x_2,x_3)&\text{ if }\,x_1>t,
             \end{array}
\right.
\text{ and } {}&&
u^1_{2,t}(x)\defeq\left\{
             \begin{array}{ll}
             u (2t-x_1,x_2,x_3) &\text{ if }\,x_1\leq  t,  \\
             u(x)&\text{ if }\,x_1>t.
             \end{array}
\right.
\end{align*}
Analogously, we define the extensions $u^3_{1,t}$ and $u^3_{2,t}$ of $u$ with respect to the hyperplane $\{x\in\R^3:x_3=t\}$. Using the identity
\begin{align*}
\widehat{K}(\xi)=\frac{4\pi}{3}\,\frac{2\xi_3^2-\xi_1^2-\xi_2^2}{|\xi|^2}=-\frac{4\pi}{3}+\frac{4\pi\xi_3^3}{|\xi|^2},
\end{align*}
we can rewrite $E(u)$ into $E(u)=E_1(u)+E_2(u)$, where
\begin{equation}
 \label{e2u}
 \begin{aligned}
  E_1(u)& \defeq \frac{1}{2}\|\nabla u\|_2^2+\frac{1}{2}\Big(\lambda_1-\frac{4\pi}{3}\lambda_2\Big)\|u\|_4^4+\frac{2}{p}\|u\|_p^p,\\
  E_2(u)& \defeq \frac{2\pi \lambda_2}{(2\pi)^3}\int_{\R^3}\frac{\xi_3^2}{|\xi|^2}\,\big|\mathcal{F}(|u|^2)(\xi)\big|^2\;d\xi.
 \end{aligned}
\end{equation}
Now let $\varepsilon>0$ be arbitrary and $u\in S\big(\frac{1}{2}(c_1+c_2)\big)$ with $E(u)\leq \gamma\big(\frac{1}{2}(c_1+c_2)\big)+\frac{\varepsilon}{2}$. We have that
\begin{equation*}
\begin{aligned}
\|u^i_{1,t}\|_p^p&=2\int_{x_i\leq t}|u|^pdx,\\
\|u^i_{2,t}\|_p^p&=2\int_{x_i\geq t}|u|^pdx
\end{aligned}
\end{equation*}
for $i\in\{1,3\}$. It also holds that $t\mapsto \|u^i_{1,t}\|_2^2$ is continuous, $\|u^i_{1,t}\|_2^2\to 0$ as $t\to-\infty$ and $\|u^i_{1,t}\|_2^2\to c_1+c_2$ as $t\to\infty$. Thus we can find some $t_1$ and $t_3$ such that
\begin{equation*}
\begin{aligned}
\|u^i_{1,t_i}\|_2^2=c_1\text{ and }\|u^i_{2,t_i}\|_p^p=c_2.
\end{aligned}
\end{equation*}
for $i\in\{1,3\}$. If $\lambda_2$ is positive, we utilize \eqref{e2u} and obtain from \cite[(4.42)]{LopesMaris2008} that
\begin{equation}\label{sym1}
\begin{aligned}
&E(u^3_{1,t_3})+E(u^3_{2,t_3})-2E(u)\\
&=-\frac{4\lambda_2}{(2\pi)^3}\int_{\R^2}\big|\xi_3'\big|\,\bigg|\int_0^\infty\mathcal{F}\big(A_3(|u|^2)\big)(\xi)\,\frac{\xi_3}{|\xi|^2}\;d\xi_3\bigg|^2\; d\xi'_3\leq 0,
\end{aligned}
\end{equation}
where $A_3(\phi)\defeq\frac{1}{2}(\phi(x_1,x_2,x_3)-\phi(x_1,x_2,-x_3))$ and $\xi_3'=(\xi_1,\xi_2)$. If $\lambda_2=0$, then it follows directly that $E(u^3_{1,t_3})+E(u^3_{2,t_3})=2E(u)$. If $\lambda_2$ is negative, we utilize \eqref{e2u} and obtain from \cite[(4.31)]{LopesMaris2008} that 
\begin{equation}\label{sym2}
\begin{aligned}
&E(u^1_{1,t_1})+E(u^1_{2,t_1})-2E(u)\\
&=\frac{4\lambda_2}{(2\pi)^3}\int_{\R^2}\frac{\xi_3^2}{\sqrt{\xi_2^2+\xi_3^2}}\,\bigg|\int_0^\infty\mathcal{F}\big(A_1(|u|^2)\big)(\xi)\,\frac{\xi_1}{|\xi|^2}\;d\xi_1\bigg|^2\; d\xi'_1\leq 0.
\end{aligned}
\end{equation}
Summing up, we obtain
\begin{align}
\gamma(c_1)+\gamma(c_2)\leq{}& E(u^i_{1,t_i})+E(u^i_{2,t_i})\leq 2E(u)\leq 2\Big(\gamma\Big(\frac{1}{2}(c_1+c_2)\Big)+\frac{\varepsilon}{2}\Big)\nonumber\\
={}&2\gamma\Big(\frac{1}{2}(c_1+c_2)\Big)+\varepsilon,
\end{align}
for all $\lambda_2\in\R$. Since $\varepsilon$ is arbitrary, we obtain \eqref{concave 1}. Together with Lemma \ref{Nonexistence of minimizers} and Lemma \ref{strict decreasing} we see that the statements
\begin{enumerate}
\item $\gamma\big(\frac{1}{2}(c_1+c_2)\big)\geq\frac{1}{2}\big(\gamma(c_1)+\gamma(c_2)\big)$ for all $c_1,c_2\in[0,\infty)$, and
\item $\gamma$ is nonincreasing on $[0,\infty)$,
\end{enumerate}
hold true. Thus, standard convex analysis implies the concavity of $\gamma$.

Suppose now that there exist some $c_1,c_2\in(0,\infty)$ such that $\gamma(c_1+c_2)=\gamma(c_1)+\gamma(c_2)$. From the concavity of $\gamma$ follows that
\begin{align*}
 \gamma(c_1)= \gamma\left(\frac{c_1}{c_1+c_2}\,(c_1+c_2)\right)\geq{}& \frac{c_1}{c_1+c_2}\gamma(c_1+c_2)+\frac{c_1}{c_1+c_2}\gamma(0)=\frac{c_1}{c_1+c_2}\gamma(c_1+c_2),\\
 \gamma(c_2)\geq{}& \frac{c_2}{c_1+c_2}\gamma(c_1+c_2).
\end{align*}
 Summing both inequalities we obtain that $\gamma(c_1+c_2)\leq \gamma(c_1)+\gamma(c_2)$ and that equality holds if and only if $\gamma(c_i)= \frac{c_i}{c_1+c_2}\gamma(c_1+c_2)$ for $i=1,2$. But this is already the case. Thus, the concave function $\gamma$ coincides with a linear function at $0,c_1,c_2$ and $c_1+c_2$, so that itself must be linear in $[0,c_1+c_2]$. This completes the proof of the lemma.
\end{proof}
\section{Existence of ground states}
This entire section is devoted to the proof of existence of standing waves and their qualitative properties. Using the geometry of the infimum function we will exclude the possibility of dichotomy for a minimizing sequence.
\begin{proof}[Proof of Theorem \ref{Theorem 1.1}]  If $\lambda_1,\lambda_2$ satisfy \eqref{case_a1} or \eqref{case_a2}, nonexistence of minimizers on $S(c)$ for all $c\in(0,\infty)$ is already proven in Lemma \ref{Nonexistence of minimizers}.

We thus assume that $\lambda_1,\lambda_2$ satisfy either \eqref{case_a3} or \eqref{case_a4}. The fact that the supremum in \eqref{c_b} is achieved, follows the continuity and  monotonicity of $\gamma$.

Next, we show the nonexistence of minimizers for $E$ on $S(c)$ for $c\in (0,c_b)$. Suppose that for some $c\in(0,c_b)$, $E$ possesses a minimizer $u$ on $S(c)$. Then $\gamma(u)=0$ so that $E(u)=0$ and $B(u)<0$. But then from proof of Lemma \ref{strict decreasing}, we have that the functions $t\mapsto E({}^tu)$ and $t \mapsto \lnorm {}^{t}u\rnorm_2^2$ are both continuous and respectively strictly decreasing and strictly increasing in $(1,\infty)$. Thus there exists $t_0>1$ such that $c<\lnorm {}^{t_0}u\rnorm_2^2<c_b$ and $E({}^{t_0}u)<0$, which contradicts the definition of $c_b$.

We next show that for all $c\in(c_b,\infty)$, $E(u)$ possesses at least one minimizer on $S(c)$, using  the classical concentration compactness lemma (\cite[Lemma III.1]{Lions1984}). Let $\{u_n\}_{n\in\N}\subset S(c)$ be a minimizing sequence, i.e., $E(u_n)=\gamma(c)+o(1)$, which is bounded in $H^1(\R^3;\C)$ due to Remark \ref{uniform boundedness}. Then one of the following three cases may occur:
\begin{enumerate}
\item \textsl{Compactness}: In this case, using a truncation argument one directly obtains that $u_n$ converges strongly to some $u$ in $L^2(\R^3;\C)$. It turns out that $u$ is also in $H^1(\R^3;\C)$ due to uniqueness of weak limits and the $H^1$-boundedness of the minimizing sequence. From the Gagliardo-Nirenberg inequality we also see that $u_n$ converges to $u$ in $L^p(\R^3;\C)$ for all $p\in[2,6)$. Using the lower semicontinuity of the $L^p$-norm and the strong $L^4$-convergence we obtain that
\begin{align*}
c_u&:=\|u\|_2^2= \lim_{n\to\infty}\|u_n\|_2^2=c,\\
E(u)&\leq \liminf_{n\to\infty} E(u_n)=\gamma(c).
\end{align*}
But since $u\in S(c)$, we have $E(u)\geq \gamma(c)$, and therefore $E(u)=\gamma(c)$. Hence $u$ is a minimizer for $E$ on $S(c)$.
\item \textsl{Vanishing}: If this were the case, then due to Lions' lemma \cite[Lemma I.1]{Lions2} we must obtain that $u_n$ converges to zero in $L^4$. But estimating like in \eqref{minimizing equal zero} leads to the contradiction $\gamma(c)\geq 0$.
\item \textsl{Dichotomy}: From the proof of Lemma \ref{maris} we obtain that
\begin{align}\label{subadd1}
\gamma(c)\leq\gamma(\alpha)+\gamma(c-\alpha).
\end{align}
for all $c\in(c_b,\infty)$ and $\alpha\in(0,c)$. But if dichotomy occurs, we must have
\begin{align}\label{subadd2}
\gamma(c)\geq\gamma(\alpha)+\gamma(c-\alpha).
\end{align}
From \eqref{subadd1} and \eqref{subadd2} follows
\begin{align}\label{subadd3}
\gamma(c)=\gamma(\alpha)+\gamma(c-\alpha).
\end{align}
Again, due to Lemma \ref{maris}, $\gamma$ must be linear on $[0,c]$. But then since $\gamma$ is constantly equal to zero on $[0,c_b]$, it follows that $\gamma(c)=0$, a contradiction.
\end{enumerate}
Since vanishing and dichotomy is ruled out, we obtain the existence of a minimizer of $E$ on $S(c)$ for all $c\in(c_b,\infty)$.
\end{proof}
\begin{proof}[Proof of Proposition \ref{Theorem 1}]
 The first two claims given by the fourth statement come directly from \cite[Theorem 4.9]{LopesMaris2008}. For $p\in(4,6)$ and $\lambda_2>0$, the problem that if all minimizers $u$ are (up to translation) axially symmetric in $x_3$ is still open. However, taking $c_1=c_2=c$, we obtain from \eqref{sym1} that if $u$ is a minimizer, then $u^3_{1,t_3}$ and $u^3_{2,t_3}$ are (up to translations) symmetric with respect to the $(x_1,x_2)$-plane and also minimizers of $E$ on $S(c)$. Using similar arguments involving the representation formulas \eqref{sym1} and \eqref{sym2} we are also able to obtain the remaining three claims given by the second statement for $p=6$.

 Finally, the second and third statements are well known results of elliptic theory. We refer to \cite[Theorem 2]{BellazziniEtAl2017} and \cite[Theorem 8.1.1]{Cazenave2003} for the respective proofs.
\end{proof}

%

\section{The case of an active trapping potential}
\paragraph{Existence of minimizers.}
Let $\{u_n\}_{n\in\N}\subset S(c)$ be a minimizing sequence, i.e., $E(u_n)=\gamma(c)+o(1)$. Then due to Lemma \ref{compactness trapping}, $\{u_n\}_{n\in\N}$ converges weakly to some $u\in S(c)$, such that $u_n\strongly u$ in $L^p$ for all $p\in[2,6)$. Using the lower semicontinuity of $L^p$-norm and strong $L^4$-convergence (which implies a.e. convergence) we immediately obtain that $u$ is a minimizer of $E$ on $S(c)$.
\paragraph{Well-posedness theory.}
For a proof of the unique local solution, we refer to \cite[Theorem 9.2.6]{Cazenave2003}. In order to show that the solution is global we estimate like in \eqref{boundedness u_4} and \eqref{boundedness nablapsi}, by replacing $\|\nabla\psi\|_2^2$ to $\|\nabla\psi\|_2^2+\|V_{ext}|u|^2\|_1$. 

\end{document}

%% file: preamble.tex
\usepackage[bbgreekl]{mathbbol}
\usepackage{amsfonts}
\usepackage{amsmath}
\usepackage{amssymb}
\DeclareSymbolFontAlphabet{\mathbb}{AMSb}
\DeclareSymbolFontAlphabet{\mathbbl}{bbold}

\usepackage[greek,english,ngerman]{babel}
\usepackage{graphicx}
\usepackage{amsmath}
\usepackage{verbatim}
\usepackage{color}
\usepackage{amstext}
\usepackage{latexsym}
\usepackage[small,it]{caption}
\usepackage{dsfont}
\usepackage{enumerate,enumitem}
\usepackage{mathtools}
\usepackage{indentfirst}
\usepackage{subcaption}
\usepackage{eurosym}
\usepackage{stmaryrd}

\usepackage[h]{esvect} 

\usepackage{bm}
\usepackage{amsthm}

\usepackage{url}

\makeatletter
\def\@endtheorem{\endtrivlist}
\makeatother
\theoremstyle{plain}
\newtheorem*{theorem*}{\protect\TheoremName}
\newtheorem*{acknowledgement*}{\protect\AcknowledgmentName}
\newtheorem*{algorithm*}{\protect\AlgorithmName}
\newtheorem*{assumption*}{\protect\AssumptionName}
\newtheorem*{assumptions*}{\protect\AssumptionsName}
\newtheorem*{axiom*}{\protect\AxiomName}
\newtheorem*{condition*}{\protect\ConditionName}
\newtheorem*{conditions*}{\protect\ConditionsName}
\newtheorem*{case*}{\protect\CaseName}
\newtheorem*{claim*}{\protect\ClaimName}
\newtheorem*{conclusion*}{\protect\ConclusionName}
\newtheorem*{conjecture*}{\protect\ConjectureName}
\newtheorem*{corollary*}{\protect\CorollaryName}
\newtheorem*{criterion*}{\protect\CriterionName}
\newtheorem*{definition*}{\protect\DefinitionName}
\newtheorem*{lemma*}{\protect\LemmaName}
\newtheorem*{notation*}{\protect\NotationName}
\newtheorem*{problem*}{\protect\ProblemName}
\newtheorem*{proposition*}{\protect\PropositionName}
\newtheorem*{solution*}{\protect\SolutionName}
\newtheorem*{summary*}{\protect\SummaryName}

\newtheorem{theorem}{\protect\TheoremName}[section]

\newtheorem{definition}[theorem]{\protect\DefinitionName}
\newtheorem{lemma}[theorem]{\protect\LemmaName}

\newtheorem{proposition}[theorem]{\protect\PropositionName}

\theoremstyle{definition}
\newtheorem*{example*}{\protect\ExampleName}
\newtheorem*{exercise*}{\protect\ExerciseName}
\newtheorem*{remark*}{\protect\RemarkName}

\newtheorem{remark}[theorem]{\protect\RemarkName}
\newtheorem*{question*}{\protect\QuestionName}
\newtheorem*{questions*}{\protect\QuestionsName}

\newcommand{\TheoremName}{}
\newcommand{\AcknowledgmentName}{}
\newcommand{\AlgorithmName}{}
\newcommand{\AssumptionName}{}
\newcommand{\AssumptionsName}{}
\newcommand{\QuestionName}{}
\newcommand{\QuestionsName}{}
\newcommand{\AxiomName}{}
\newcommand{\ConditionName}{}
\newcommand{\ConditionsName}{}
\newcommand{\CaseName}{}
\newcommand{\ClaimName}{}
\newcommand{\ConjectureName}{}
\newcommand{\ConclusionName}{}
\newcommand{\CorollaryName}{}
\newcommand{\CriterionName}{}
\newcommand{\DefinitionName}{}
\newcommand{\LemmaName}{}
\newcommand{\NotationName}{}
\newcommand{\ProblemName}{}
\newcommand{\PropositionName}{}
\newcommand{\SolutionName}{}
\newcommand{\SummaryName}{}
\newcommand{\ExampleName}{}
\newcommand{\RemarkName}{}
\newcommand{\ExerciseName}{}
\addto\captionsenglish{%
  \renewcommand{\TheoremName}{Theorem}
  \renewcommand{\AcknowledgmentName}{Acknowledgments}
  \renewcommand{\AlgorithmName}{Algorithm}
  \renewcommand{\AssumptionName}{Assumption}
  \renewcommand{\AssumptionsName}{Assumptions}
  \renewcommand{\QuestionName}{Question}
  \renewcommand{\QuestionsName}{Questions}
  \renewcommand{\AxiomName}{Axiom}
  \renewcommand{\ConditionName}{Condition}
  \renewcommand{\ConditionsName}{Conditions}
  \renewcommand{\CaseName}{Case}
  \renewcommand{\ClaimName}{Claim}
  \renewcommand{\ConjectureName}{Conjecture}
  \renewcommand{\ConclusionName}{Conclusion}
  \renewcommand{\CorollaryName}{Corollary}
  \renewcommand{\CriterionName}{Criterion}
  \renewcommand{\DefinitionName}{Definition}
  \renewcommand{\LemmaName}{Lemma}
  \renewcommand{\NotationName}{Notation}
  \renewcommand{\ProblemName}{Problem}
  \renewcommand{\PropositionName}{Proposition}
  \renewcommand{\SolutionName}{Solution}
  \renewcommand{\SummaryName}{Summary}
  \renewcommand{\ExampleName}{Example}%
  \renewcommand{\ExerciseName}{Exercise}%
  \renewcommand{\RemarkName}{Remark}%
}
\addto\captionsngerman{%
  \renewcommand{\TheoremName}{Theorem}
  \renewcommand{\AcknowledgmentName}{Anerkennungen}
  \renewcommand{\AlgorithmName}{Algorithmus}
  \renewcommand{\AssumptionName}{Voraussetzung}
  \renewcommand{\AssumptionsName}{Voraussetzungen}
  \renewcommand{\AxiomName}{Axiom}
  \renewcommand{\ConditionName}{Bedingung}
  \renewcommand{\ConditionsName}{Bedingungen}
  \renewcommand{\CaseName}{Fall}
  \renewcommand{\ClaimName}{Behauptung}
  \renewcommand{\ConjectureName}{Vermutung}
  \renewcommand{\ConclusionName}{Folgerung}
  \renewcommand{\CorollaryName}{Korollar}
  \renewcommand{\CriterionName}{Kriterium}
  \renewcommand{\DefinitionName}{Definition}
  \renewcommand{\LemmaName}{Lemma}
  \renewcommand{\NotationName}{Notation}
  \renewcommand{\ProblemName}{Problem}
  \renewcommand{\PropositionName}{Satz}
  \renewcommand{\SolutionName}{L\"osung}
  \renewcommand{\SummaryName}{Zusammenfassung}
  \renewcommand{\ExampleName}{Beispiel}%
  \renewcommand{\ExerciseName}{\"Ubung}%
  \renewcommand{\RemarkName}{Bemerkung}%
}
\addto\captionsgreek{%
  \renewcommand{\TheoremName}{Je\'wrhma}
  \renewcommand{\AcknowledgmentName}{Euqarist\'ies}
  \renewcommand{\AlgorithmName}{Alg\'orijmos}
  \renewcommand{\AssumptionName}{Up\'ojesh}
  \renewcommand{\AssumptionsName}{Upoj\'eseis}
  \renewcommand{\AxiomName}{Ax\'iwma}
  \renewcommand{\ConditionName}{Pro\"up\'ojesh}
  \renewcommand{\ConditionsName}{Pro\"upoj\'eseis}
  \renewcommand{\CaseName}{Per\'iptwsh}
  \renewcommand{\ClaimName}{Isqurism\'os}
  \renewcommand{\ConjectureName}{Eikas\'ia}
  \renewcommand{\ConclusionName}{Sump\'erasma}
  \renewcommand{\CorollaryName}{P\'orisma}
  \renewcommand{\CriterionName}{Krit\'hrio}
  \renewcommand{\DefinitionName}{Orism\'os}
  \renewcommand{\LemmaName}{L\'hmma}
  \renewcommand{\NotationName}{Shmeiograf\'ia}
  \renewcommand{\ProblemName}{Pr\'oblhma}
  \renewcommand{\PropositionName}{Pr\'otash}
  \renewcommand{\SolutionName}{L\'ush}
  \renewcommand{\SummaryName}{Per\'ilhyh}
  \renewcommand{\ExampleName}{Par\'adeigma}%
  \renewcommand{\ExerciseName}{\'Askhsh}%
  \renewcommand{\RemarkName}{Parat\'hrhsh}%
}

\makeatletter
\renewenvironment{proof}[1][\proofname]{\par
\pushQED{\hfill$\blacksquare$}%
\normalfont \topsep6\p@\@plus6\p@\relax
\trivlist
\item\relax
{\bfseries
#1\@addpunct{.}}\hspace\labelsep\ignorespaces
}{%
\popQED\endtrivlist\@endpefalse
}
\makeatother

\newcommand{\strongly}{%
  \mathrel{
    \resizebox{1.4em}{0.88ex}{$\rightarrow$}
}}

\newcommand{\wstar}{%
  \mathrel{\vbox{\offinterlineskip\ialign{%
    \hfil##\hfil\cr
    $\scriptstyle*\,$\cr
    \noalign{\kern 0.12ex}
    \resizebox{1.4em}{0.88ex}{$\rightharpoonup$}\cr
}}}}

\newcommand{\twoweakly}{%
  \mathrel{\vbox{\offinterlineskip\ialign{%
    \hfil##\hfil\cr
    $\scriptstyle2\,$\cr
    \noalign{\kern 0.12ex}
    \resizebox{1.4em}{0.88ex}{$\rightharpoonup$}\cr
}}}}

\newcommand{\twowstar}{%
  \mathrel{\vbox{\offinterlineskip\ialign{%
    \hfil##\hfil\cr
    $\scriptstyle2*\,$\cr
    \noalign{\kern 0.12ex}
    \resizebox{1.4em}{0.88ex}{$\rightharpoonup$}\cr
}}}}

\newcommand{\twostrongly}{%
  \mathrel{\vbox{\offinterlineskip\ialign{%
    \hfil##\hfil\cr
    $\scriptstyle2\,$\cr
    \noalign{\kern 0.12ex}
    \resizebox{1.4em}{0.88ex}{$\rightarrow$}\cr
}}}}

\def\XXint#1#2#3{{\setbox0=\hbox{$#1{#2#3}{\int}$ }\hspace{0.17em}
\vcenter{\hbox{$#2#3$ }}\kern-.585\wd0}}

\newcommand{\lnorm}{\left\|}
\newcommand{\rnorm}{\right\|}



\makeatletter
\newcommand{\defeq}{\mathrel{\hspace{-0.1ex}\mathrel{\rlap{\raisebox{0.3ex}{$\m@th\cdot$}}\raisebox{-0.3ex}{$\m@th\cdot$}}\hspace{-0.73ex}\resizebox{0.55em}{0.84ex}{$=$}\hspace{0.67ex}}\hspace{-0.25em}}
\makeatother
\makeatletter
\newcommand{\eqdef}{\hspace{-0.25em}\mathrel{\hspace{0.67ex}\resizebox{0.55em}{0.84ex}{$=$}\hspace{-0.73ex}\mathrel{\rlap{\raisebox{0.3ex}{$\m@th\cdot$}}\raisebox{-0.3ex}{$\m@th\cdot$}}\hspace{-0.1ex}}}
\makeatother
\usepackage{tikz}


\usepackage[colorinlistoftodos]{todonotes}
\usepackage[author={A.S.},color=yellow]{pdfcomment}
\DTMsetregional 

\newcommand{\refpart}[2]{\hyperref[#1]{\ref*{#1}--\textit{#2.}}}